\documentclass[11pt]{article}%
\usepackage{amssymb}
\usepackage{amsfonts}
\usepackage{amsmath}
\usepackage{amssymb}
\usepackage{color}
\usepackage{graphicx}
\usepackage{caption}%
\usepackage{hyperref}
\usepackage{float}
\usepackage{subcaption}
\usepackage{academicons}

\usepackage{tabularx}

\providecommand{\U}[1]{\protect\rule{.1in}{.1in}}
\newtheorem{theorem}{Result}

\newtheorem{corollary}[theorem]{Corollary}

\newtheorem{example}[theorem]{Example}
\newtheorem{proposition}[theorem]{Proposition}
\newtheorem{remark}[theorem]{Remark}
\newtheorem{definition}[theorem]{Definition}
\newenvironment{proof}[1][Proof]{\noindent \textbf{#1.} }{\hspace{16cm} \rule{0.5em}{0.5em}}
\newcommand\blfootnote[1]{%
	\begingroup
	\renewcommand\thefootnote{}\footnote{#1}%
	\addtocounter{footnote}{-1}%
	\endgroup
}

\begin{document}
	
	\title{\textbf{Robust statistical inference for accelerated life-tests with one-shot devices under log-logistic distributions  }}
	\author{María González$^1$, Mar\'ia Jaenada$^1$ and Leandro Pardo$^1$}
	\date{ }
	\maketitle

\indent $^1$Department of Statistics and O.R., Complutense University of Madrid, Madrid, Spain.\\

\blfootnote{Correspondence:  M.Jaenada: mjaenada@ucm.es	and L. Pardo: lpardo@mat.ucm.es}

\begin{abstract}
A one-shot device is a unit that operates only once, after which it is either destroyed or needs to be rebuilt. For this type of device, the operational status can only be assessed at a specific inspection time, determining whether failure occurred before or after it. Consequently, lifetimes are subject to left- or right-censoring.
One-shot devices are usually highly reliables. To analyze the reliability of such products, an accelerated life test (ALT) plan is typically employed by subjecting the devices to increased levels of stress factors, thus allowing life characteristics observed under high-stress conditions to be extrapolated to normal operating conditions. By accelerating the degradation process, ALT significantly reduces both the time required for testing and the associated experimental costs.

Recently, robust inferential methods have gained considerable interest in statistical analysis. Among them, weighted minimum density power divergence estimators (WMDPDEs) are widely recognized for their robust statistical properties with small loss of efficiency. In this work, robust WMDPDE and associated statistical tests are developed under a log-logistic lifetime distribution with multiple stresses. Explicit expressions for the estimating equations and asymptotic distribution of the estimators are obtained. Further,  a Monte Carlo simulation study is presented to evaluate the performance of the WMDPDE in practical applications.

\end{abstract}

\section{Introduction}
One-shot devices are products designed for single-use applications, such as
automobile airbags, fuel injectors, missiles, fire extinguishers, explosive
charges, munitions, magnetorheological fluids, and thermal batteries. The
practical utility of one-shot devices is highlighted in the work of Fan et
al. (2009), who applied one-shot device testing data to analyze highly
reliable electro-explosive devices using a Bayesian approach with
exponential, normal, and beta prior distributions. Their importance across
various fields is further explored in the book by Balakrishnan et al. (2021)
and the review by Balakrishnan and Ling (2023). A key characteristic of
one-shot devices is that they are destroyed immediately after use, meaning
their actual lifetime remains unknown. The only available information is
whether the device failed before or after a given test time. Consequently,
experimental results for these devices are binary, classified as either
right-censored (failure) or left-censored (success). Notably, this setup
also has practical applications in survival analysis, particularly in
medicine, where it is known as current status data.

On the other hand, most of the products manufactured nowadays are of high
quality, so they will have a long service life. If the products are tested
under normal conditions, the time to failure of the products will be quite
long; therefore, the test duration will also be long, and costly. Besides, a
reliability experiment with scarce failures observed will result in
insignificant inferential results. Therefore, inducing failures in shorter
experimental time would be desirable for accurate inference. To reduce the
experimental time and cost, Accelerated Life-Tests (ALTs) plans are often
used to examine the reliability of such reliable products. ALTs shortens the
life of products by increasing the levels of certain stressors influencing
the degradation of the product. Some example of stress factors are air
pressure, temperature, humidity and voltage, that can artificially
controlled in an laboratory.  However, the main purpose of reliability
analysis is understand the lifetime behaviour or some related
characteristics under normal operating conditions. Then, after estimating
the model parameters from data collected under high-stress conditions,
results should be extrapolated to normal operating conditions; see Meeter
and Meeker (1994) and Meeker et al. (1998).

Parametric inference usually assumes a certain family of distributions
depending on certain parameters to model the data. In ALTs, the model
parameters of the lifetime distribution must be related to stress factors,
so that measurements taken during the experiment can then be extrapolated
back to the expected performance under normal operating conditions. In this
paper, we use a log-linear relationship between the stress level and the
log-logistic distribution to relate the log-logistic parameters with the
levels of stress.

Likelihood inference for one-shot device testing data under several popular
lifetime distributions has been extensively studied in the literature (see
Balakrishnan and Ling (2012a,b,2013,2014). These results are essentially
based on maximum likelihood estimator (MLE). The MLE is known to be
asymptotically efficient and, specifically, a BAN (Best Asymptotically
Normal) estimator. However, it is sensitive to outliers and sampling errors,
which highlights the need for robust alternatives. To address the issue of
robustness without a significant loss of efficiency, Balakrishann et al.
(2019a, b; 2020a, b; 2921; 2022a, b) developed and studied the weighted
minimum density power divergence estimators (WMDPDE) in the context of
one-shot devices with different lifetime distributions under a single and
multiple stress factors, as well as Wald-type tests based on WMDPDE for some
of these families. Baghel and Mondal (2024a,b) presented also some results
of robust inference for one-shot devices. A review about WMDPDE as well as
Wald-type tests based on them can be seen in Balakrishnan et al (2021).

In this paper we consider robust statistical inference for ALTs with
one-shot devices under log-logistic lifetime distribution. The Log-Logistic
probability distribution is a classical lifetime distribution in
reliability. It has the advantage (like the Weibull and exponential model)
of having simple algebraic expressions for reliability and hazard function.
Its failure rate function is not strictly monotonic therefore, it is
flexible to express the failure rate of products. It is therefore more
convenient than the log-normal distribution in handling censored data, while
providing a good approximation to it except in the extreme tails, see Du,
and Gui (2019), Bennett (983), Granzotto, and Louzada (2015) and Serkan
(2012).

\bigskip 
Section 2 presents the model and the assumed distribution, along
with the Maximum Likelihood Estimator (MLE). In Section 3, we introduce the
Weighted Minimum Density Power Divergence Estimator (WMDPDE) and derive the
asymptotic distribution for lifetimes under a log-logistic distribution.
Section 4 addresses robustness by analyzing the influence function, while
Section 5 introduces the Wald-type and Rao-type tests for hypothesis testing
within the assumed model. Finally, in an empirical context, Section 6
provides a numerical analysis, including Monte Carlo simulations and
estimations based on a real dataset.
\section{Model description and maximum likelihood estimator}
\label{sec:descripcion}
Let us assume that the one-shot device data from a reliability testing experiment are stratified into \( I  \) testing conditions with $K_i$  devices placed under the \(i\)-th testing condition \(\ (i = 1, \ldots, I)\) . Each testing condition is defined by the combination of \(J\) stress factors \( x_{ij}, j = 0, 1, \ldots, J \) at fixed stress levels. Devices functioning status are observed at fixed inspection times \( \tau_i \) \( i = 1, \ldots, I \), which are pre-specified in advance. The vector of testing conditions for each stratum \( i=1,...,I \) will be denoted by \( \boldsymbol{x}_i \) \(= (x_{i0},...,x_{iJ}) \). After the inspection, the number of failures under the $i$-th testing condition among the $K_i$ devices is recorded as $n_i.$    \( i = 1, \ldots, I \)  . Table \ref{table:datos_dispositivos} presents the observed data.
\begin{table}[h]
	\captionsetup{justification=centering, skip=5pt, position=above}
	\caption{Data for one-shot devices with multiple stress levels observed at different inspection times}
	\centering
	\begin{tabular}{|c|c|c|c|c|c|c|}
		\hline
		Condition & Inspection Time & Failures & Stress Factor 1 & ... & Stress Factor $J$ & Devices \\
		\hline
		1 & $\tau_1$ & $n_1$ & $x_{11}$ & ... & $x_{1J}$ & $K_1$ \\
		2 & $\tau_2$ & $n_2$ & $x_{21}$ & ... & $x_{2J}$ & $K_2$ \\
		... & ... & ... & ... & ... & ... & ... \\
		$j$ & $\tau_j$ & $n_j$ & $x_{j1}$ & ... & $x_{jJ}$ & $K_j$ \\
		... & ... & ... & ... & ... & ... & ... \\
		$I$ & $\tau_I$ & $n_I$ & $x_{I1}$ & ... & $x_{IJ}$ & $K_I$ \\
		\hline
	\end{tabular}
	\label{table:datos_dispositivos}
\end{table}

Let \( G \) be the true cumulative distribution function underlying the lifetimes of the devices and \( g \) the corresponding probability density function.
We shall assume that the life of the devices, \( T_{ik}, (i = 1, 2, \ldots, I) \), under the testing condition \( i \), follows a log-logistic distribution with probability density function
\begin{equation}
	f_{\alpha_i,\beta_i}(t, \boldsymbol{x}_i) = \frac{\beta_i \alpha_i^{\beta_i} t^{\beta_i - 1}}{(t^{\beta_i} + \alpha_i^{\beta_i})^2}, \quad t > 0 \ (\alpha_i, \beta_i > 0)
	\label{eq:Fdensity}
\end{equation}
and cumulative distribution function
\begin{equation}
	F_{\alpha_i,\beta_i}(t, \boldsymbol{x}_i) = \frac{t^{\beta_i}}{t^{\beta_i} + \alpha_i^{\beta_i}}, \quad t > 0 \ (\alpha_i, \beta_i > 0)
	\label{eq:Fdistribution}
\end{equation}
where $\alpha_i > 0$ is the scale parameter and corresponds to the median of the distribution, and $\beta_i > 0$ is the parameter that controls the shape of the distribution, which is unimodal for $\beta_i > 1$. 

The survival or reliability function and the hazard rate function can be straightforward derived from the distribution function described in equation (\ref{eq:Fdistribution}) as follows
\begin{equation}
	R_{\alpha_i,\beta_i}(t, \boldsymbol{x}_i) =  \frac{\alpha_i^{\beta_i}}{t^{\beta_i} + \alpha_i^{\beta_i}} \quad t > 0  \hspace{2mm} \text{   and   }  \hspace{2mm} h_{\alpha_i,\beta_i}(t, \boldsymbol{x}_i) =  \frac{\beta_i t^{\beta_i - 1}}{t^{\beta_i} + \alpha_i^{\beta_i}} \text{ ,}  (\alpha_i, \beta_i > 0)
	\label{eq:rloglogis}
\end{equation}
respectively.
\begin{figure}
	\centering
	\includegraphics[width=1\linewidth]{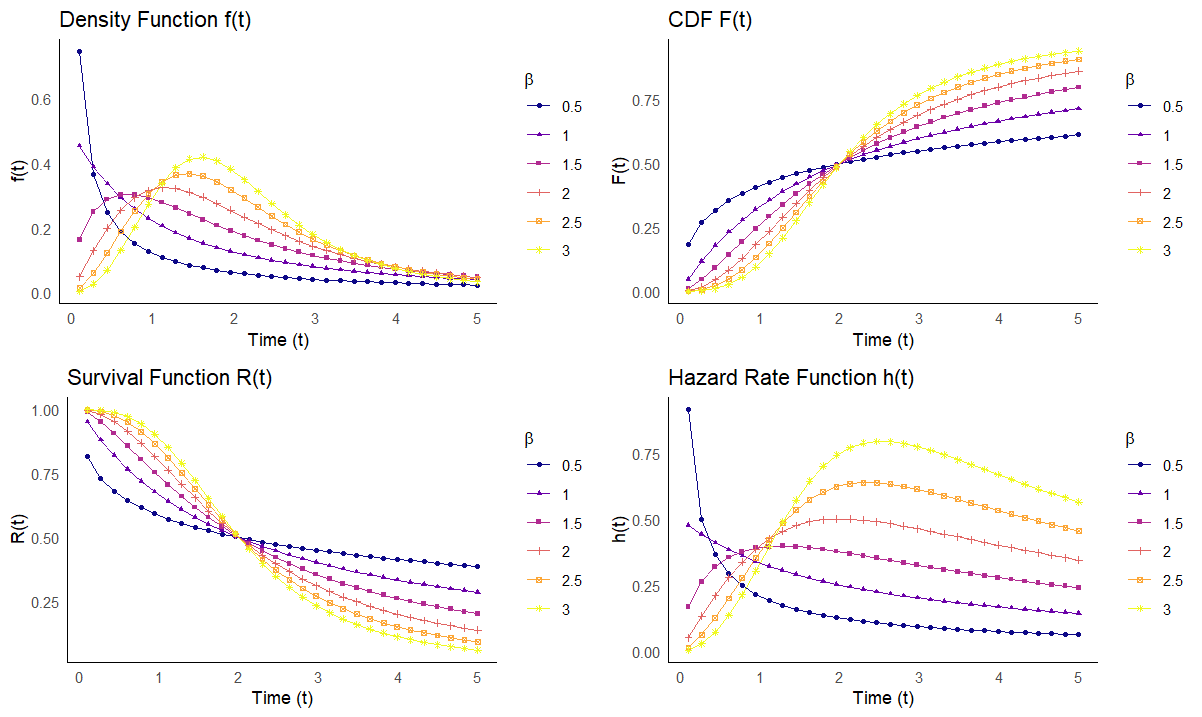}
	\caption{Log-Logistic Distribution Functions with Fixed Scale Parameter \(\alpha = 2\) and Varying Shape Parameter}
	\label{fig:distrib}
\end{figure}

Noteworthy, the hazard ratio is monotonically decreasing for \(\beta_i \leq 1\), while for \(\beta_i > 1\) it increases to a maximum atained at 
\( t = \left( \frac{\alpha_i - 1}{\beta_i} \right)^{\frac{1}{\alpha_i}} \)
and then decreases to zero as \( t \to \infty \). This characteristic proves valuable in scenarios where the event risk varies, displaying non-constant or non-monotonic behaviour. This property can be observed in Figure \ref{fig:distrib}, as well as the density function, cumulative distribution and survival function. 
The log-logistic distribution shares similarities with both the Weibull and log-normal distributions, which are commonly used in survival analysis. The Weibull distribution is suited for modelling monotonic hazard rates, which either increase or decrease depending on its shape parameter. On the other hand, both the log-logistic and log-normal distributions can represent non-monotonic hazard rates, with the hazard rate increasing to a peak and then decreasing. An advantage of the log-logistic distribution is its mathematically simpler expressions for the hazard and survival functions, making it more straightforward to use in practical applications.

Because the primary goal of ALT is to extrapolate results under normal operating conditions, for ALT tests it is necessary to relate the lifetime distribution of the devices the stress levels at which units are tested. Here, we assume a log-linear link function between the stress levels and the log-logistic distribution parameters as follows,  
\begin{equation*}
	\alpha_i = \exp \left( \sum_{j=0}^{J} a_j x_{ij} \right) \quad \text{and} \quad \beta_i = \exp \left( \sum_{j=0}^{J} b_j x_{ij} \right) \text{ ,}
\end{equation*}
allowing $\alpha_i$ and $\beta_i$ to differ depending on the stress conditions. Therefore, the parameters to be estimated in our model are
\(
\boldsymbol{\theta} = ( a_j, b_j, j = 0, \ldots, J ) \text{ .}
\)
In the current framework, the observed data represent counts of failures under different stress conditions. Given the discrete nature of grouped the data, they should modelled by a discrete probability distributions, such as the Bernoulli distribution.

Let us first introduce the theoretical probability vectors of Bernoulli distributions as
\begin{equation}
	\boldsymbol{\pi}_i(\boldsymbol{\theta}) = (\pi_{i1}(\boldsymbol{\theta}), \pi_{i2}(\boldsymbol{\theta}))^T = (F_{\alpha_i, \beta_i}(\tau_i, \boldsymbol{x}_i), R_{\alpha_i, \beta_i}(\tau_i, \boldsymbol{x}_i))^{T}, \quad i = 1, \ldots, I,
\end{equation}
where the first element represents the theoretical probability of failure before the inspection time, and the second represents the theoretical reliability function. 

Now we shall define the empirical probability vectors:
\begin{equation}
	\hat{\boldsymbol{p}}_i = (\hat{p}_{i1}, \hat{p}_{i2})^T  =( \frac{n_i}{K_i},  1 - \frac{n_i}{K_i})^T, \quad i = 1, \ldots, I,
\end{equation}
where \(\hat{p}_{i1} = \frac{n_i}{K_i} \) is the observed proportion of failures in the \(i\)-th group, and \(\hat{p}_{i2} = 1 - \hat{p}_{i1}\) represents the observed proportion of successes. The empirical probability vector \(\hat{\boldsymbol{p}}_i\) constitutes a non-parametric estimator of the theoretical probabilities \(\boldsymbol{\pi}_i(\boldsymbol{\theta})\).

With the previous notation, the likelihood function based on the observed data is given by
\begin{equation}
	L(\boldsymbol{\theta}) = \prod_{i=1}^{I}  [F_{\boldsymbol{\theta}}(\tau_i, \boldsymbol{x}_i)]^{n_i} [R_{\boldsymbol{\theta}}(\tau_i, \boldsymbol{x}_i)]^{K_i - n_i} = \prod_{i=1}^{I}  \pi_{i1}(\boldsymbol{\theta})^{n_i} \pi_{i2}(\boldsymbol{\theta})^{K_i - n_i} =  \prod_{i=1}^{I}  \frac{t^{\beta_i \cdot n_i}}{(t^{\beta_i} + \alpha_i^{\beta_i})^{K}} \cdot \alpha_i^{\beta_i(K-n_i)} \text{,}
	\label{eq:verosim}
\end{equation}
and the Maximum Likelihood Estimator (MLE), \(\hat{\boldsymbol{\theta}} = ( \hat{a}_j, \hat{b}_j, j = 0, \ldots, J )\) is then given by
\(
\hat{\boldsymbol{\theta}} = \arg \max_{\boldsymbol{\theta}} \text{ } log L(\boldsymbol{\theta})
\).
The MLE is recognized as asymptotically efficient and a BAN (\textit{Best Asymptotically Normal}) estimator. Nevertheless, its performance may be influenced by the presence of outliers in the sample or by inaccuracies arising from sampling errors, a property referred to as robustness.

An alternative characterization of the MLE, \(\hat{\boldsymbol{\theta}}  \), can be formulated using the Kullback-Leibler divergence. This perspective serves as the foundation for deriving the Weighted Minimum Density Power Divergence Estimator (WMDPDE) within the model framework. 
The Kullback-Leibler divergence between the empirical probability vector \(\hat{\boldsymbol{p}}_i\) and the theoretical probability vector for the testing condition $i$, \(\boldsymbol{\pi_i}(\boldsymbol{\theta})\), is given by
\allowdisplaybreaks
\[
D_{KL}(\boldsymbol{\hat{p_i}}, \boldsymbol{\pi_i}(\boldsymbol{\theta})) = \hat{p}_{i1} \log \frac{\hat{p}_{i1}}{\pi_{i1}(\boldsymbol{\theta})} + \hat{p}_{i2} \log \frac{\hat{p}_{i2}}{\pi_{i2}(\boldsymbol{\theta})} 
= \frac{n_i}{K_i} \log \frac{\frac{n_i}{K_i}}{F_{\boldsymbol{\theta}}(\tau_i, \boldsymbol{x}_i)} + \frac{K_i - n_i}{K_i} \log \frac{\frac{K_i - n_i}{K_i}}{1 - F_{\boldsymbol{\theta}}(\tau_i, \boldsymbol{x}_i)} \text{ .}
\]
The weighted Kullback-Leibler divergence ($WD_{KL}$) between the probability distributions $\hat{\boldsymbol{p}}_i$ and $\boldsymbol{\pi_i}(\boldsymbol{\theta})$, $i = 1, \ldots, I$, is defined as:
\[ WD_{KL}(\boldsymbol{\theta}) = \sum_{i=1}^{I} \frac{K_i}{K} D_{KL}(\hat{\boldsymbol{p}}_i , \boldsymbol{\pi_i}(\boldsymbol{\theta})) \] 
where \( K = K_1 + K_2 + \cdots + K_I \). It is not difficult to verify that   
\begin{equation}
	WD_{KL}(\boldsymbol{\theta}) = c - \frac{1}{K} \log L(\boldsymbol{\theta})
	\label{eq:relacion_wdkl}
\end{equation}
where \( c \) is a constant that does not depend on \( \boldsymbol{\theta} \) and \(L(\boldsymbol{\theta})\) was defined in (\ref{eq:verosim}). Based on (\ref{eq:relacion_wdkl}) we can define the MLE by
\(
\hat{\boldsymbol{\theta}} = \arg \min_{\boldsymbol{\theta}} \text{WD}_{KL} (\boldsymbol{\theta}) \text{ .}
\)

\section{The minimum density power divergence for one-shot device with log-logistic lifetime}
The relationship observed in the previous section between the MLE and the Kullback-Leibler divergence suggests defining a class of estimators using a divergence measure  distinct from the Kullback-Leibler divergence, and ideally, these estimators would address the lack of robustness associated with the MLE. Given the good performance of the density power divergence (DPD) for both destructive and non-destructive one-shot device models (see Balakrishnan et al., 2019a and b; 2020a and 2024a and b, Baghel and Kondal (2023,2024a,b)). We use the DPD approach here.

The DPD measure was considered for the first time in Basu et al., 1998. 
The DPD with tuning parameter \(\gamma >0\) between the probability vectors \(\hat{p_i}\) and \( \pi_i(\boldsymbol{\theta})\) is given by
\begin{equation}
	D_\gamma(\hat{p_i}, \pi_i(\boldsymbol{\theta})) = \left( \pi_{i1}^{\gamma+1}(\boldsymbol{\theta}) + \pi_{i2}^{\gamma+1}(\boldsymbol{\theta}) \right) - \frac{\gamma + 1}{\gamma} \left( \hat{p}_{i1} \pi_{i1}(\boldsymbol{\theta})^\gamma + \hat{p}_{i2} \pi_{i2}(\boldsymbol{\theta})^\gamma \right) + \frac{1}{\gamma} \left( \hat{p}_{i1}^{\gamma+1} + \hat{p}_{i2}^{\gamma+1} \right).
	\label{eq:dpd_divergence}
\end{equation}
As \(\gamma \to 0\), the divergence measure converges to the Kullback-Leibler divergence, while for \(\gamma = 1\), it reduces to the euclidean distance between the probability vectors \(\hat{\boldsymbol{p}}_i\) and \(\pi_i(\boldsymbol{\theta})\). The tuning parameter \(\gamma > 0\) controls the trade-off between robustness and asymptotic efficiency of the parameter estimates.
From the DPDs, \(D_\gamma(\hat{p_i}, \pi_i(\boldsymbol{\theta}))\), the weighted density power divergence (WDPD) is defined as
\[
\text{WD}_\gamma (\theta) = \sum_{i=1}^{I} \frac{K_i}{K} D_\gamma(\hat{p_i}, \pi_i(\boldsymbol{\theta})),
\]
where \(K = K_1 + K_2 + \cdots + K_I\). Thus, the weighted minimum density power divergence estimator (WMDPDE) is defined by
\(
\hat{\boldsymbol{\theta}}_\gamma = \arg \min_{\boldsymbol{\theta}} \text{WD}_\gamma (\boldsymbol{\theta})
\).
The term \(\frac{1}{\gamma} \left( \hat{p}_{i1}^{\gamma+1} + \hat{p}_{i2}^{\gamma+1} \right)\) in  (\ref{eq:dpd_divergence}) plays no role in the minimization with respect to \(\boldsymbol{\theta}\). Thus, the problem reduces to minimizing, with respect to \(\boldsymbol{\theta}\), the objective function
\begin{equation}
	WD_\gamma^*(\boldsymbol{\theta}) = \sum_{i=1}^{I} \frac{K_i}{K} D^*_\gamma (\hat{p_i}, \pi_i(\boldsymbol{\theta})) 
	\label{eq:fobj} \text{ ,}
\end{equation}
being 
\begin{equation}
	D_\gamma^*(\hat{p_i}, \pi_i(\boldsymbol{\theta})) = \left( \pi_{i1}^{\gamma+1}(\boldsymbol{\theta}) + \pi_{i2}^{\gamma+1}(\boldsymbol{\theta}) \right) - \frac{\gamma + 1}{\gamma} \left( \hat{p}_{i1} \pi_{i1}(\boldsymbol{\theta})^\gamma + \hat{p}_{i2} \pi_{i2}(\boldsymbol{\theta})^\gamma \right)  \text{ .}
	\label{eq:medida_f}
\end{equation}  
The following theorem presents the system of equations that must be solved to obtain the WMDPDE under log-logistic lifetimes.

\begin{theorem} The WMDPDE of \(\boldsymbol{\theta}\) with tuning parameter \(\gamma \geq 0\) for log-logistic lifetimes, \(\hat{\boldsymbol{\theta}}_\gamma\), can be obtained as the solution to the following system of equations:
	\begin{equation}
		\sum_{i=1}^{I} \left( -\beta_i \boldsymbol{x}_i, \log \frac{\tau_i}{\alpha_i} \beta_i \boldsymbol{x}_i \right) \left( K_i \frac{\tau_i^{\beta_i}}{\tau_i^{\beta_i} + \alpha_i^{\beta_i}} - n_i \right) \left( \frac{\tau_i^{\beta_i \gamma} \alpha_i^{\beta_i} + \alpha_i^{\beta_i \gamma} \tau_i^{\beta_i}}{(\tau_i^{\beta_i} + \alpha_i^{\beta_i})^{\gamma + 1}} \right) = \boldsymbol{\mathbf{0}}_{2(J+1)}  \text{ .}
		\label{eq:ecuacionesloglin} 
	\end{equation}
	where \(\boldsymbol{\mathbf{0}}_{2(J+1)}\) is the null column vector of dimension \(2(J+1)\).
	\label{th:ecuaciones}
\end{theorem}
\begin{proof} (See Appendix: Section \ref{suplmaterial_1})
\end{proof}
Before deriving the asymptotic distribution of \(\hat{\boldsymbol{\theta}}_\gamma\), the following lemma is introduced:

\begin{proposition}
	\label{prop:eclibro}
	Let \(\boldsymbol{\theta}_0\) denote the true value of the parameter and \(F_{\boldsymbol{\theta}}(\tau_i, x_i)\) the cdf modelling the one-shot devices under test. Under some regularity conditions, the asymptotic distribution of the WMDPDE, \( \hat{\boldsymbol{\theta}}_\gamma \), is given by the following expression:
	\[ 
	\sqrt{K} (\hat{\boldsymbol{\theta}}_\gamma - \boldsymbol{\theta}_0) \underset{K \rightarrow \infty}{\overset{\mathcal{L}}{\to}} \mathcal{N}\biggr( \boldsymbol{0}_{2(J+1)}, \boldsymbol{J}_\gamma^{-1}(\boldsymbol{\theta}_0) \boldsymbol{K}_\gamma (\boldsymbol{\theta}_0) \boldsymbol{J}_\gamma^{-1}(\boldsymbol{\theta}_0)\biggr),
	\]    
	where
	\[
	\boldsymbol{J}_\gamma (\boldsymbol{\theta}) = \sum_{i=1}^{I} \frac{K_i}{K}  \biggr( F_{\boldsymbol{\theta}}(\tau_i, \boldsymbol{x}_i)^{\gamma-1} + R_{\boldsymbol{\theta}}(\tau_i, \boldsymbol{x}_i)^{\gamma-1} \biggr) \frac{ \partial F_{\boldsymbol{\theta}}(\tau_i, \boldsymbol{x}_i)}{\partial \boldsymbol{\theta} } \frac{ \partial F_{\boldsymbol{\theta}}(\tau_i, \boldsymbol{x}_i)}{\partial \boldsymbol{\theta}^T},
	\]
	\[
	\boldsymbol{K}_{\gamma}(\boldsymbol{\theta}) = \sum_{i=1}^{I}  \frac{K_i}{K} F_{\boldsymbol{\theta}}(\tau_i, \boldsymbol{x}_i) R_{\boldsymbol{\theta}}(\tau_i, \boldsymbol{x}_i)  \biggr( F_{\boldsymbol{\theta}}(\tau_i, \boldsymbol{x}_i)^{\gamma-1} + R_{\boldsymbol{\theta}}(\tau_i, \boldsymbol{x}_i)^{\gamma-1} \biggr)^2 \frac{ \partial F_{\boldsymbol{\theta}}(\tau_i, \boldsymbol{x}_i)}{\partial \boldsymbol{\theta} } \frac{ \partial F_{\boldsymbol{\theta}}(\tau_i, \boldsymbol{x}_i)}{\partial \boldsymbol{\theta}^T}.
	\]
	\label{prop:libro}
\end{proposition}
\begin{proof}(See Appendix: Section \ref{suplmaterial_2})
\end{proof}
The following theorem presents explicit expressions of matrices \(\boldsymbol{J}_{\gamma}(\boldsymbol{\theta}) \) and \(\boldsymbol{K}_{\gamma}(\boldsymbol{\theta}) \).
\begin{theorem}
	Let \(\boldsymbol{\theta}_0 = (a^0_j, b^0_j, j = 1, \ldots, J)\) be the true value of the parameter \(\boldsymbol{\theta}\). The asymptotic distribution of the WMDPDE, \(\hat{\boldsymbol{\theta}}_\gamma = (\hat{a}_j, \hat{b}_j, j = 0, \ldots, J)\), assuming the lifetimes follow a log-logistic distribution, is given by the following expression:
	\[ \sqrt{K} (\hat{\boldsymbol{\theta}}_\gamma - \boldsymbol{\theta}_0) \underset{K \rightarrow \infty}{\overset{\mathcal{L}}{\to}} \mathcal{N}\biggr( \boldsymbol{0}_{2(J+1)}, \boldsymbol{J}_\gamma^{-1}(\boldsymbol{\theta}_0) \boldsymbol{K}_\gamma (\boldsymbol{\theta}_0) \boldsymbol{J}_\gamma^{-1}(\boldsymbol{\theta}_0)\biggr),
	\]    
	being
	\[
	\boldsymbol{J}_\gamma (\boldsymbol{\theta}) = \sum_{i=1}^{I} \frac{K_i}{K} \boldsymbol{M}_i \left( R_{\alpha_i, \beta_i} (\tau_i, x_i) F_{\alpha_i, \beta_i} (\tau_i, x_i) \right)^2 \left( F_{\alpha_i, \beta_i} (\tau_i, x_i)^{\gamma - 1} + R_{\alpha_i, \beta_i} (\tau_i, x_i)^{\gamma - 1} \right) 
	\]
	and 
	\[
	\boldsymbol{K}_\gamma(\boldsymbol{\theta}) = \sum_{i=1}^{I} \frac{K_i}{K} \left( F_{\alpha_i, \beta_i}(\tau_i, \boldsymbol{x_i})^{\gamma-1} + R_{\alpha_i, \beta_i}(\tau_i, \boldsymbol{x_i})^{\gamma-1} \right)^2 F_{\alpha_i, \beta_i}(\tau_i, \boldsymbol{x_i})^3 \cdot R_{\alpha_i, \beta_i}(\tau_i, \boldsymbol{x_i})^3  \cdot \boldsymbol{M}_i
	\]
	where
	\allowdisplaybreaks
	\[
	\boldsymbol{M}_i = \begin{pmatrix}
		\beta_i^2 \boldsymbol{x_i} \boldsymbol{x_i}^T & -\log \frac{\tau_i}{\alpha_i} \beta_i^2 \boldsymbol{x_i} \boldsymbol{x_i}^T \\
		-\log \frac{\tau_i}{\alpha_i} \beta_i^2 \boldsymbol{x_i} \boldsymbol{x_i}^T & \left( \log \frac{\tau_i}{\alpha_i} \beta_i \right)^2 \boldsymbol{x_i} \boldsymbol{x_i}^T 
	\end{pmatrix} \text{ .}
	\]
	\label{th:asim_distr}
\end{theorem}
\begin{proof}(See Appendix: Section \ref{sec:th:asim_distr})
	
\end{proof}
\begin{corollary}
	The asymptotic distribution of the MLE of \(\boldsymbol{\theta}\), \(\hat{\boldsymbol{\theta}}_{\gamma=0}\), is given by:
	\[
	\sqrt{K}(\hat{\boldsymbol{\theta}}_{\gamma=0} - \boldsymbol{\theta}_0) \underset{K \rightarrow \infty}{\overset{\mathcal{L}}{\to}}
	N\biggr(\boldsymbol{0}_{2(J+1)}, \boldsymbol{I_F}(\boldsymbol{\theta}_0)^{-1}\biggr),
	\]
	where
	\begin{align*}
		\boldsymbol{I_F}(\boldsymbol{\theta})  &=  \sum_{i=1}^{I} K_i \boldsymbol{M}_i \left( R_{\alpha_i, \beta_i} (\tau_i, \boldsymbol{x}_i) F_{\alpha_i, \beta_i} (\tau_i, \boldsymbol{x}_i) \right)^2 \left( F_{\alpha_i, \beta_i} (\tau_i, \boldsymbol{x}_i)^{-1} + R_{\alpha_i, \beta_i} (\tau_i, \boldsymbol{x}_i)^{-1} \right)
	\end{align*}
	is the Fisher information matrix, being
	\[
	\boldsymbol{M}_i = \begin{pmatrix}
		\beta_i^2 \boldsymbol{x}_i \boldsymbol{x}_i^T & -\log \frac{\tau_i}{\alpha_i} \beta_i^2 \boldsymbol{x}_i \boldsymbol{x}_i^T \\
		-\log \frac{\tau_i}{\alpha_i} \beta_i^2 \boldsymbol{x}_i \boldsymbol{x}_i^T & \left( \log \frac{\tau_i}{\alpha_i} \beta_i \right)^2 \boldsymbol{x}_i \boldsymbol{x}_i^T
	\end{pmatrix}.
	\]
	Moreover, it can be observed that
	\[
	\boldsymbol{J}_{\gamma=0}(\boldsymbol{\theta}) = \boldsymbol{K}_{\gamma=0}(\boldsymbol{\theta}) = \boldsymbol{I_F}(\boldsymbol{\theta}).
	\]
	\label{prop:emvdistr}
\end{corollary}
\section{Robustness Analysis}
Now we study the robustness of the WMDPDE, \( \hat{\boldsymbol{\theta}}_\gamma\), by using the influence function (IF) of the WMDPDE. The concept of IF of an estimator was introduced by Hampel [19] and used from then onwards to evaluate the robustness of an estimator.
The approach of Balakrishnan et al. (2021a) will be followed. 
Given a general distribution function \(G\), a statistical functional \(\boldsymbol{T}(G)\) is any real-valued function of \(G\), denoted as \( \boldsymbol{\theta} = \boldsymbol{T}(G).
\)
For any estimator defined in terms of a statistical functional \(\boldsymbol{T}(G)\) of the true distribution \(G\), its influence function (IF) is defined as:
\[
\text{IF}(t_0, \boldsymbol{T}, G) = \lim_{\epsilon \to 0} \frac{\boldsymbol{T}(G_\epsilon) - \boldsymbol{T}(G)}{\epsilon} = \left. \frac{\partial \boldsymbol{T}(G_\epsilon)}{\partial \epsilon} \right|_{\epsilon=0},
\]
where \(G\) is the true distribution, \(G_\epsilon\) is the contaminated distribution function, and a contamination point at \(t_0\) is considered. 

Let us consider the \(i\)-th testing condition, in which \(K_i\) one-shot devices are observed, and \(n_i\) is the number of failures observed. We denote by \(G_i\) the true cumulative distribution function of the Bernoulli random variable that generates the \(n_i\) failures. The probability mass function associated to \(G_i\) is denoted by \(\boldsymbol{p}_i\). On the other hand, let $F$ be the distribution function of the Bernoulli random variable with a probability of success equal to \(F_{\boldsymbol{\theta}}(\tau_i, \boldsymbol{x}_i)\) and the corresponding probability mass function is denoted by \(\boldsymbol{\pi}_{i}(\boldsymbol{\theta})\). The following vector notation is introduced, where $\otimes$ is the Kronecker product:
\[
\boldsymbol{G} = (G_1 \otimes \mathbf{1}_{K_1}, \ldots, G_I \otimes \mathbf{1}_{K_I})^T \quad \text{y} \quad \boldsymbol{F}_{\boldsymbol{\theta}} = (F_{1;\boldsymbol{\theta}} \otimes \mathbf{1}_{K_1}, \ldots, F_{I;\boldsymbol{\theta}} \otimes \mathbf{1}_{K_I})^T.
\]
The contaminated distribution function in the direction of the point $\boldsymbol{t}  = (t_{11}, \dots, t_{1K_1}, \dots, t_{I1}, \dots, t_{IK_I})$ of a random variable with distribution function $G$ is given by \(
G_\epsilon = (1 - \epsilon) G + \epsilon \Delta_{\boldsymbol{t}}
\),   where \(\epsilon\) is the contamination proportion and \(\Delta_{\boldsymbol{t}}\) denotes the cumulative distribution function of a degenerate variable at the contamination point \(\boldsymbol{t}\).

To derive the influence function of the WMDPDE, it is necessary to define the statistical functional \(\boldsymbol{T}_\gamma(G)\) associated with the estimator as the minimizer of the weighted sum of the DPDs between the true and model probability mass functions. The minimum density power functional is the minizer of:
\begin{equation}
	H_\gamma(\boldsymbol{\theta}) = \sum_{i=1}^I \frac{K_i}{K} \left\{ \sum_{y \in \{0,1\}} \left[ \boldsymbol{\pi}_i^{\gamma+1}(y, \boldsymbol{\theta}) - \frac{\gamma + 1}{\gamma} \boldsymbol{\pi}_i^\gamma(y, \boldsymbol{\theta}) \boldsymbol{p}_i(y) \right] \right\},
	\label{Tgamma}
\end{equation}
where \( \boldsymbol{p}_i(y) \) is the probability mass function associated with \( G_i \), and
\[
\boldsymbol{\pi}_i(y, \boldsymbol{\theta}) = y F_{\boldsymbol{\theta}}(\tau_i, \boldsymbol{x}_i) + (1 - y) R_{\boldsymbol{\theta}}(\tau_i, \boldsymbol{x}_i), \quad y \in \{0, 1\}.
\]
Choosing \( \boldsymbol{p}_i(y) \equiv \boldsymbol{\pi}_i(y, \boldsymbol{\theta}) \), the expression \eqref{Tgamma} is minimized at \( \boldsymbol{\theta} = \boldsymbol{\theta}_0 \), which implies that the WMDPDE functional \( \boldsymbol{T}_\gamma(G) \) is Fisher-consistent. 
Furthermore, we have
\begin{equation}
	\frac{\partial H_\gamma(\boldsymbol{\theta})}{\partial \boldsymbol{\theta}} = \sum_{i=1}^I \frac{K_i}{K} \left\{ \sum_{y \in \{0,1\}} \left[ \boldsymbol{\pi}_i^\gamma(y, \boldsymbol{\theta}) \frac{\partial \boldsymbol{\pi}_i(y, \boldsymbol{\theta})}{\partial \boldsymbol{\theta}} - \boldsymbol{\pi}_i^{\gamma-1}(y, \boldsymbol{\theta}) \frac{\partial \boldsymbol{\pi}_i(y, \boldsymbol{\theta})}{\partial \boldsymbol{\theta}} \boldsymbol{p}_i(y) \right] \right\} = \boldsymbol{0}_{2(J+1)}.
	\label{eq:ecuacionesIF}
\end{equation}

To compute the influence function (IF) of the WMDPDE at \(F_{\boldsymbol{\theta}}\) with respect to the \(k\)-th element of the \(i_0\)-th group of observations, we replace \(\boldsymbol{\theta}\) in expression \eqref{eq:ecuacionesIF} with
\[
\boldsymbol{\theta}^{i_0}_{\epsilon} = \boldsymbol{T}_\gamma(G_1 \otimes \boldsymbol{1}_{K_1}^T, \ldots, G_{i_0-1} \otimes \boldsymbol{1}_{K_{i_0-1}}^T, G_{i_0, \epsilon} \otimes \boldsymbol{1}_{K_{i_0}}^T, G_{i_0+1} \otimes \boldsymbol{1}_{K_{i_0+1}}^T, \ldots, G_I \otimes \boldsymbol{1}_{K_I}^T),
\]
where \(G_{i_0, \epsilon}\) is the distribution function associated with the probability mass function
\[
\boldsymbol{p}_{i_0, \epsilon, k}(y) = (1 - \epsilon) \boldsymbol{\pi}_{i_0}(y, \boldsymbol{\theta}_0) + \epsilon \delta_{t_{i_0, k}}(y),
\]
and
\(
\delta_{t_{i_0, k}}(y) = y \delta_{t_{i_0, k}}^{(1)} + (1 - y) \delta_{t_{i_0, k}}^{(2)},
\)
with \(\delta_{t_{i_0, k}}^{(1)}\) being the degenerate function at point \((i_0, k)\), \(\delta_{t_{i_0, k}}^{(2)} = (1 - \delta_{t_{i_0, k}}^{(1)})\), and
\[
\boldsymbol{p}_i(y) =
\begin{cases}
	\boldsymbol{\pi}_i(y, \boldsymbol{\theta}_0) & \text{if } i \ne i_0, \\
	\boldsymbol{p}_{i_0, \epsilon, k}(y) & \text{if } i = i_0.
\end{cases}
\]

In Balakrishnan et al. (2021a), the following two results were established.

\begin{theorem}	\label{prop:robustness1}
	Let $\boldsymbol{T}_\gamma(G)$ be the statistical functional corresponding to the WMDPDE estimator, obtained by minimizing expression (\ref{Tgamma}). Then, the IF for one-shot device testing with multiple stress factors with respect to the \(k\)-th observation of the \(i_0\)-th group is given by
	\begin{align}
		\text{IF}(t_{i_0,k}, \boldsymbol{T}_\gamma, F_{\boldsymbol{\theta}_0}) &= \boldsymbol{J}_\gamma^{-1}(\boldsymbol{\theta}_0) \frac{K_{i_0}}{K} \left( \frac{\partial F_{\boldsymbol{\theta}_0}(\tau_i, \boldsymbol{x}_i)}{\partial \boldsymbol{\theta}} \right)_{\boldsymbol{\theta}=\boldsymbol{\theta}_0}  \left( F_{\boldsymbol{\theta}_0}(\tau_i, \boldsymbol{x}_i)^{\gamma-1} + R_{\boldsymbol{\theta}_0}(\tau_i, \boldsymbol{x}_i)^{\gamma-1} \right) \left( F_{\boldsymbol{\theta}_0}(\tau_i, \boldsymbol{x}_i) - \Delta_{t_{i_0}}^{(1)} \right)
	\end{align}
	where \(\Delta_{t_{i_0}}^{(1)}\) is the cumulative distribution function of a degenerate variable at point \((t_{i_0})\).
\end{theorem}
\begin{theorem}\label{prop:robustez2}
	Let $\boldsymbol{T}_\gamma(G)$ be the statistical functional corresponding to the WMDPDE estimator, obtained by minimizing expression (\ref{Tgamma}). Then, the IF for one-shot device testing with multiple stress factors with respect to all observations is given by:
	\begin{align}
		\text{IF}(t, \boldsymbol{T}_\gamma, F_{\boldsymbol{\theta}_0}) &= \boldsymbol{J}_\gamma^{-1}(\boldsymbol{\theta}_0) \sum_{i=1}^I \frac{K_{i}}{K} \left( \frac{\partial F_{\boldsymbol{\theta}}(\tau_i, \boldsymbol{x}_i)}{\partial \boldsymbol{\theta}} \right)_{\boldsymbol{\theta}=\boldsymbol{\theta}_0}  \left( F_{\boldsymbol{\theta}}(\tau_i, \boldsymbol{x}_i)^{\gamma-1} + R_{\boldsymbol{\theta}}(\tau_i, \boldsymbol{x}_i)^{\gamma-1} \right) \left( F_{\boldsymbol{\theta}_0}(\tau_i, \boldsymbol{x}_i) - \Delta_{t_{i}}^{(1)} \right)
	\end{align}
	where \(\Delta_{t_{i}}^{(1)} = \sum_{k=1}^{K_I} \Delta_{t_{i, k}}^{(1)}\).
\end{theorem}

Under the detailed conditions and assuming log-logistic lifetimes, from equations (\ref{eq:derFaj}) and (\ref{eq:derFbj}), we have:
\begin{align*}
	\frac{ \partial F_{\alpha_i, \beta_i}(\tau_i, \boldsymbol{x}_i)}{\partial \boldsymbol{\theta} }   &= \begin{pmatrix}
		-\frac{ \alpha_i^{\beta_i} \tau_i^{\beta_i}}{(\tau_i^{\beta_i} + \alpha_i^{\beta_i})^2} \cdot \beta_i \boldsymbol{x}_i \text{,} & \frac{\tau_i^{\beta_i} \alpha_i^{\beta_i} \log \frac{\tau_i}{\alpha_i}}{(\tau_i^{\beta_i} + \alpha_i^{\beta_i})^2} \cdot \beta_i \boldsymbol{x}_i 
	\end{pmatrix}^T   \begin{pmatrix}
		- \beta_i \boldsymbol{x}_i \text{,} &  \log \frac{\tau_i}{\alpha_i} \cdot \beta_i \boldsymbol{x}_i
	\end{pmatrix}^T 
	F_{\alpha_i, \beta_i}(\tau_i, \boldsymbol{x}_i) \cdot R_{\alpha_i, \beta_i}(\tau_i, \boldsymbol{x}_i).
\end{align*}

Based on the previous results, the following theorem is established.
\begin{theorem}
	Let $\boldsymbol{T}_\gamma(G)$ be the statistical functional corresponding to the WMDPDE estimator under log-logistic lifetimes, obtained by minimizing the expression \eqref{Tgamma}. Then,
	\begin{enumerate}
		\item The IF for one-shot device testing with multiple stress factors with respect to the \(k\)-th observation of the \(i_0\)-th group, assuming log-logistic lifetime distributions, is given by:
		\begin{align*}
			\text{IF}(t_{i_0,k}, \boldsymbol{T}_\gamma, F_{\boldsymbol{\theta}_0}) &= \boldsymbol{J}_\gamma^{-1}(\boldsymbol{\theta}_0) \frac{K_{i_0}}{K} 
			\begin{pmatrix}
				- \beta_{i_0} \boldsymbol{x}_{i_0} \text{,} &  \log \frac{\tau_{i_0}}{\alpha_{i_0}} \cdot \beta_{i_0} \boldsymbol{x}_{i_0}
			\end{pmatrix}^T \\
			&\quad \cdot F_{\alpha_{i_0}, \beta_{i_0}}(\tau_{i_0}, \boldsymbol{x}_{i_0}) \cdot R_{\alpha_{i_0}, \beta_{i_0}}(\tau_{i_0}, \boldsymbol{x}_{i_0}) \\
			&\quad \left( F_{\alpha_{i_0}, \beta_{i_0}}(\tau_{i_0}, \boldsymbol{x}_{i_0})^{\gamma-1} + R_{\alpha_{i_0}, \beta_{i_0}}(\tau_{i_0}, \boldsymbol{x}_{i_0})^{\gamma-1} \right) \\
			&\quad \cdot \left( F_{\alpha_{i_0}, \beta_{i_0}}(\tau_{i_0}, \boldsymbol{x}_{i_0}) - \Delta_{t_{i_0}}^{(1)} \right),
		\end{align*}
		where \(\Delta_{t_{i_0}}^{(1)}\) is the cumulative distribution function of a degenerate variable at the point \((t_{i_0})\).
		
		\item The IF for one-shot device testing with multiple stress factors with respect to all observations is given by:
		\begin{align*}
			\text{IF}(t_{i_0,k}, \boldsymbol{T}_\gamma, F_{\boldsymbol{\theta}_0}) &= \boldsymbol{J}_\gamma^{-1}(\boldsymbol{\theta}_0) \sum_{i=1}^I \frac{K_{i}}{K} 
			\begin{pmatrix}
				- \beta_{i_0} \boldsymbol{x}_{i_0} \text{,} &  \log \frac{\tau_{i_0}}{\alpha_{i_0}} \cdot \beta_{i_0} \boldsymbol{x}_{i_0}
			\end{pmatrix}^T \\
			&\quad \cdot F_{\alpha_{i_0}, \beta_{i_0}}(\tau_{i_0}, \boldsymbol{x}_{i_0}) \cdot R_{\alpha_{i_0}, \beta_{i_0}}(\tau_{i_0}, \boldsymbol{x}_{i_0}) \\
			&\quad \left( F_{\alpha_{i_0}, \beta_{i_0}}(\tau_{i_0}, \boldsymbol{x}_{i_0})^{\gamma-1} + R_{\alpha_{i_0}, \beta_{i_0}}(\tau_{i_0}, \boldsymbol{x}_{i_0})^{\gamma-1} \right) \\
			&\quad \cdot \left( F_{\alpha_{i_0}, \beta_{i_0}}(\tau_{i_0}, \boldsymbol{x}_{i_0}) - \Delta_{t_{i}}^{(1)} \right),
		\end{align*}
		where \(\Delta_{t_{i}}^{(1)} = \sum_{k=1}^{K_I} \Delta_{t_{i, k}}^{(1)}\).
	\end{enumerate}
	
\end{theorem}

\section{Testing hypothesis}
We are interested in testing a general null hypothesis that restricts the parameter space $\boldsymbol{\Theta}$ to a proper subset $\Theta_0$ of $\Theta = \{\boldsymbol{\theta} : \boldsymbol{\theta} = (a_j, b_j, j = 0, \ldots, J)\} \subset \mathbb{R}^{2(J+1)}$, that is, testing
\begin{equation}
	H_0: \boldsymbol{\theta} \in \Theta_0 \text{ against } H_1: \boldsymbol{\theta} \notin \Theta_0.
	\label{eq:H0Wald}
\end{equation}

The restricted parameter space $\Theta_0$ is defined by a set of $r$ restrictions expressed as
\begin{equation}
	\boldsymbol{m}(\boldsymbol{\theta}) = \boldsymbol{\mathbf{0}}_r,
	\label{eq:g}
\end{equation}
on $\Theta$, where $\boldsymbol{m}: \mathbb{R}^{2(J+1)} \rightarrow \mathbb{R}^r$ is a vector function such that the matrix of dimension $2(J+1) \times r$
\begin{equation*}
	\boldsymbol{M}(\boldsymbol{\theta}) = \frac{\partial \boldsymbol{m}^T(\boldsymbol{\theta})}{\partial \boldsymbol{\theta}}
\end{equation*}
exists, is continuous at $\boldsymbol{\theta}$, and satisfies that $\text{rank}(\boldsymbol{M}(\boldsymbol{\theta})) = r$. $\boldsymbol{\mathbf{0}}_r$ denotes the null vector of dimension $r$. 
Therefore, $\Theta_0 = \{ \boldsymbol{\theta} \in \Theta : \boldsymbol{m}(\boldsymbol{\theta})= \boldsymbol{\mathbf{0}}_r \}$. This setup includes the significance test of stress factors on each parameter, such as testing $H_0: a_j = b_j = 0$ for specific components of $\boldsymbol{\theta}$. Such tests have practical utility in determining the importance of individual stress factors and their contributions to the overall reliability or survival model.
We now aim to generalize the Wald and Rao test statistics to obtain robust versions using the WMDPDE. The main advantage of these robust versions is their resilience to outliers, ensuring that decisions are not strongly influenced by atypical observations. 

\subsection{Wald-Type Test Statistics}
\begin{definition}
	The Wald-type test statistics for testing  (\ref{eq:H0Wald}) based on the WMDPDE is defined as
	\allowdisplaybreaks
	\begin{equation*}
		W_K(\hat{\boldsymbol{\theta}}_\gamma) = K  \boldsymbol{m}(\hat{\boldsymbol{\theta}_\gamma})^T \left( \boldsymbol{M}(\hat{\boldsymbol{\theta}}_\gamma)^T \boldsymbol{\Sigma}(\hat{\boldsymbol{\theta}}_\gamma) \boldsymbol{M}(\hat{\boldsymbol{\theta}}_\gamma) \right)^{-1}  \boldsymbol{m}(\hat{\boldsymbol{\theta}_\gamma})
	\end{equation*}
	where $\boldsymbol{\Sigma}(\hat{\boldsymbol{\theta}}_\gamma) = \boldsymbol{J}_\gamma^{-1}(\hat{\boldsymbol{\theta}}_\gamma) \boldsymbol{K}_\gamma(\hat{\boldsymbol{\theta}}_\gamma) \boldsymbol{J}_\gamma^{-1}(\hat{\boldsymbol{\theta}}_\gamma)$.
	\label{def:Wald}
\end{definition}
This expression represents a generalization of the classical Wald test statistic. When \(\gamma = 0\), the WMDPDE coincides with the MLE, and the test statistic reduces to the classical Wald statistic, which can be written as:
\[
W_K(\hat{\boldsymbol{\theta}}_{\gamma=0}) = K  \boldsymbol{m}(\hat{\boldsymbol{\theta}}_{\gamma=0})^T \left( \boldsymbol{M}(\hat{\boldsymbol{\theta}}_{\gamma=0})^T \boldsymbol{I_F}(\hat{\boldsymbol{\theta}}) \boldsymbol{M}(\hat{\boldsymbol{\theta}}_{\gamma=0}) \right)^{-1}  \boldsymbol{m}(\hat{\boldsymbol{\theta}}_{\gamma=0}),
\]
where \(\boldsymbol{I_F}(\boldsymbol{\theta})\) is the Fisher information matrix defined in Corollary \ref{prop:emvdistr}.
Based on the asymptotic distribution of \(  \hat{\boldsymbol{\theta}}_\gamma \) can be seen established that,
\begin{equation}
	W_K(\hat{\boldsymbol{\theta}}_\gamma) \underset{K \rightarrow \infty}{\overset{ \mathcal{L}}{\to}}  \chi_r^2. 
	\label{distrWald}
\end{equation}
According to (\ref{distrWald}), we reject the null hypothesis stated in (\ref{eq:H0Wald}), at level $\alpha$, if:
\begin{equation}
	W_K(\hat{\boldsymbol{\theta}}_\gamma) > \chi_{r,\alpha}^2 \text{ .}
\end{equation}
The Wald-type test statistics are Fraser consistent, meaning that:
\begin{equation}
	\lim_{K \to \infty} \Pr\left(W_K(\hat{\boldsymbol{\theta}}_\gamma) > \chi_{r,\alpha}^2\right) = 1,
\end{equation}

\begin{example}
	Let us assume that the objective is to evaluate whether the stress factors $l$ and $r$ ($l < r$) have a significant relationship with the device's useful life. To address this, we shall consider testing the following hypothesis:
	\begin{equation*}
		H_0: a_l = a_r = 0, b_l = b_r = 0.
	\end{equation*}
	Using the notation presented in equation (\ref{eq:g}), we have:
	\begin{equation*}
		H_0: \boldsymbol{m}(\boldsymbol{\theta}) = \boldsymbol{0}_4,
	\end{equation*}
	where:
	\begin{equation*}
		\boldsymbol{m}(\boldsymbol{\theta}) = (a_l, a_r, b_l, b_r)^T.
	\end{equation*}
	In this case, we have:
	\begin{equation*}
		\boldsymbol{M}(\boldsymbol{\theta}) = 
		\begin{scriptsize}
			\begin{array}{c}
				\left(
				\begin{array}{cccccccccccccccc}
					1 &  \cdots \cdots & l+1 & \cdots \cdots & r+1 &  \cdots \cdots & J+1 & \cdots \cdots & J+1+r & \cdots \cdots & J+1+l &  \cdots \cdots & 2(J+1) \\
					\hline \\
					0 & \cdots \cdots \hspace{3mm}  & 1 & \cdots \cdots  & 1 & \cdots \cdots & 0 & \cdots \cdots & 0 &  \cdots \cdots & 0 & \cdots \cdots & 0 \\
					0 & \cdots \cdots  & 0 & \cdots \cdots & 0 & \cdots \cdots & 0 & \cdots \cdots & 1 & \cdots \cdots & 1 & \cdots \cdots & 0  
				\end{array} 
				\right).
			\end{array}
		\end{scriptsize}
	\end{equation*}
	
	According to this matrix and the expressions for $\boldsymbol{J}_\gamma(\boldsymbol{\theta})$ and $\boldsymbol{K}_\gamma(\boldsymbol{\theta})$, we reject the null hypothesis if $W_K(\hat{\boldsymbol{\theta}}_\gamma) > \chi_{4,\alpha}^2$.
\end{example}

\subsection{Rao-Type Test Statistic}
In order to define Rao-type tests for composite null hypotheses, it is necessary to introduce the restricted WMDPDE (RWMDPDE), which is the WMDPDE estimator under a restricted parameter space.
\begin{definition}
	The RWMDPDE of $\boldsymbol{\theta} \in \Theta_0 = \{\boldsymbol{\theta} \in \Theta : \boldsymbol{m}(\boldsymbol{\theta}) = \boldsymbol{0}_r\}$ is defined by
	\begin{equation*}
		\tilde{\boldsymbol{\theta}}_\gamma    = \arg \min_{\boldsymbol{\theta} \in \Theta_0} WD_{PDD}^\gamma(\boldsymbol{\theta}) = \arg \min_{\boldsymbol{\theta} \in \Theta_0} \sum_{i=1}^I \frac{K_i}{K}D^*_{\gamma} (\hat{\boldsymbol{p}}_i, \boldsymbol{\pi}_i(\boldsymbol{\theta})). 
	\end{equation*}
\end{definition}

From Theorem \ref{th:ecuaciones}, if we denote
\begin{equation}
	\boldsymbol{U}_{\gamma}(\boldsymbol{\theta}) = \sum_{i=1}^{I} \left( -\beta_i \boldsymbol{x}_i, \log \frac{\tau_i}{\alpha_i} \beta_i \boldsymbol{x}_i \right) \left( K_i \frac{\tau_i^{\beta_i}}{\tau_i^{\beta_i} + \alpha_i^{\beta_i}} - n_i \right) \left( \frac{\tau_i^{\beta_i \gamma} \alpha_i^{\beta_i} + \alpha_i^{\beta_i \gamma} \tau_i^{\beta_i}}{(\tau_i^{\beta_i} + \alpha_i^{\beta_i})^{\gamma + 1}} \right),
	\label{eq:u}
\end{equation}
the RWMDPDE is obtained as the solution of the following system of $2(J+1)+r$ equations:
\begin{equation}
	\begin{cases}
		\boldsymbol{U}_\gamma(\boldsymbol{\theta}) + \boldsymbol{M}(\boldsymbol{\theta}) \boldsymbol{\lambda}_\gamma = \boldsymbol{0}_{2(J+1)} \\
		\boldsymbol{m}(\boldsymbol{\theta}) = \boldsymbol{0}_r
	\end{cases}
\end{equation}
where $\lambda$ denotes the Lagrange multipliers.
Martin (2023) presented the asymptotic distribution of the RWMDPDE when the observations are independent but not identically distributed, as a generalization of the result presented in Basu et al. (2022) for independent and identically distributed random variables. The following result is obtained by applying the above findings to the context of one-shot devices, assuming lifetimes follow a log-logistic distribution.

\begin{theorem}
	Assuming that all the matrices defined above exist, are finite, and that the true distribution belongs to the model with the true parameter $\boldsymbol{\theta}_0 \in \Theta$, the RWMDPDE of $\boldsymbol{\theta}$ exhibits the following properties:
	
	a) The RWMDPDE estimation equations have a sequence of consistent roots $\tilde{\boldsymbol{\theta}}_\gamma$ such that
	\begin{equation*}
		\tilde{\boldsymbol{\theta}}_\gamma    \underset{K \rightarrow \infty}{\overset{ \mathcal{P}}{\to}} 
		\boldsymbol{\theta}_0.
	\end{equation*}

	b) The asymptotic null distribution of the RWMDPDE is given by,
	\begin{equation*}
		\sqrt{K} (\tilde{\boldsymbol{\theta}}_\gamma   - \boldsymbol{\theta}_0) \underset{K \rightarrow \infty}{\overset{ \mathcal{L}}{\to}}  N \biggr(\boldsymbol{0}_{2(J+1)}, \boldsymbol{\Sigma}_\gamma(\boldsymbol{\theta}_0)\biggr), 
	\end{equation*}
	where,
	\begin{equation*}
		\boldsymbol{\Sigma}_\gamma(\boldsymbol{\theta}_0) = \boldsymbol{P}_\gamma(\theta_0) \boldsymbol{K}_{\gamma}(\boldsymbol{\theta}_0) \boldsymbol{P}_\gamma(\boldsymbol{\theta}_0)
	\end{equation*}
	with
	\begin{equation*}
		\boldsymbol{P}_\gamma(\boldsymbol{\theta}_0) = \boldsymbol{J}_{\gamma}^{-1}(\boldsymbol{\theta}_0) - \boldsymbol{Q}_\gamma(\boldsymbol{\theta}_0) \boldsymbol{M}(\boldsymbol{\theta}_0)^T \boldsymbol{J}_{\gamma}^{-1}(\boldsymbol{\theta}_0)
	\end{equation*}
	and $$\boldsymbol{Q}_\gamma(\boldsymbol{\theta}_0) = \boldsymbol{J}_{\gamma}^{-1}(\boldsymbol{\theta}_0) \boldsymbol{M}(\boldsymbol{\theta}_0) \biggr( \boldsymbol{M}(\boldsymbol{\theta}_0)^T \boldsymbol{J}_{\gamma}^{-1}(\boldsymbol{\theta}_0) \boldsymbol{M}(\boldsymbol{\theta}_0) \biggr)^{-1}.$$ 
\end{theorem}
\begin{corollary}
	For \(\gamma = 0\), we have the MLE, and in this case, the asymptotic distribution of \(\tilde{\boldsymbol{\theta}}_{\gamma=0}\) is obtained considering the results of Corollary \ref{prop:emvdistr}. It follows that
	\[
	\sqrt{K} (\tilde{\boldsymbol{\theta}}_{\gamma=0}  - \boldsymbol{\theta}_0) \underset{K \rightarrow \infty}{\overset{ \mathcal{L}}{\to}}  N\biggr(\boldsymbol{0}_{2(J+1)}, \boldsymbol{S}(\boldsymbol{\theta}_0)\biggr),
	\]
	where \(\boldsymbol{S}(\boldsymbol{\theta}_0) =
	\boldsymbol{I_F}(\boldsymbol{\theta}_0)^{-1} \biggr(\boldsymbol{I} - \boldsymbol{M}(\boldsymbol{\theta}_0) \biggr(\boldsymbol{M}(\boldsymbol{\theta}_0)^T \boldsymbol{I_F}(\boldsymbol{\theta}_0)^{-1} \boldsymbol{M}(\boldsymbol{\theta}_0)\biggr)^{-1} \boldsymbol{M}(\boldsymbol{\theta}_0)^T\biggr) \boldsymbol{I_F}(\boldsymbol{\theta}_0)^{-1}.
	\)
\end{corollary}
The following result will be important to obtain the asymptotic distribution of the Rao-type tests.
\begin{proposition}
	The asymptotic distribution of \(\boldsymbol{U}_{\gamma}(\tilde{\boldsymbol{\theta}})\), defined in equation (\ref{eq:u}), is
	\[
	\sqrt{K} \boldsymbol{U}_{\gamma}(\tilde{\boldsymbol{\theta}}_{\gamma}) \underset{K \rightarrow \infty}{\overset{ \mathcal{L}}{\to}} N\biggr(\boldsymbol{0}_{2(J+1)},  \boldsymbol{K}_{\gamma}(\boldsymbol{\theta}_0) \biggr),
	\]
	and therefore,
	\[
	K\boldsymbol{U}_{\gamma}(\tilde{\boldsymbol{\theta}}_{\gamma})^T \boldsymbol{Q}_{\gamma}(\boldsymbol{\theta}_0) \biggr(\boldsymbol{Q}_{\gamma}(\boldsymbol{\theta}_0)^T \boldsymbol{K}_{\gamma}(\boldsymbol{\theta}_0) \boldsymbol{Q}_{\gamma}(\boldsymbol{\theta}_0)\biggr)^{-1} \boldsymbol{Q}_{\gamma}(\boldsymbol{\theta}_0)^T \boldsymbol{U}_{\gamma}(\tilde{\boldsymbol{\theta}}_{\gamma}) \underset{K \rightarrow \infty}{\overset{ \mathcal{L}}{\to}} \chi^2_r.
	\]
	\label{result:rao}
\end{proposition}
\begin{proof}(See Appendix: Section \ref{sec:proof:result:rao})
	
\end{proof}
According to this result, we define the Rao-type test statistics for testing the null hypothesis (\ref{eq:H0Wald}).
\begin{definition}
	The Rao-type test statistics for testing (\ref{eq:H0Wald}) and based on the RWMDPDPE is defined by
	\begin{equation}
		R_K(\tilde{\boldsymbol{\theta}}_{\gamma}) = K \boldsymbol{U}_\gamma^T (\tilde{\boldsymbol{\theta}}_{\gamma}) \boldsymbol{Q}_\gamma (\tilde{\boldsymbol{\theta}}_{\gamma}) \left( \boldsymbol{Q}_\gamma^T (\tilde{\boldsymbol{\theta}}_{\gamma}) \boldsymbol{K}_\gamma (\tilde{\boldsymbol{\theta}}_{\gamma}) \boldsymbol{Q}_\gamma (\tilde{\boldsymbol{\theta}}_{\gamma}) \right)^{-1} \boldsymbol{Q}_\gamma^T (\tilde{\boldsymbol{\theta}}_{\gamma}) \boldsymbol{U}_\gamma (\tilde{\boldsymbol{\theta}}_{\gamma}).
		\label{def:Rao} 
	\end{equation}
\end{definition}
\begin{remark}
	The above expression depends only on the restricted estimator, so it is not necessary to compute unrestricted estimators. Therefore, for simple null hypotheses, \( H_0 : \boldsymbol{\theta} = \boldsymbol{\theta}_0 \), no estimator is needed to define the test, which is one of its most significant advantages.
\end{remark}

The following theorem provides the asymptotic distribution of the Rao-type test statistics $R_K(\tilde{\boldsymbol{\theta}}_\gamma)$.
\begin{theorem}
	The asymptotic distribution of the Rao-type test statistics for testing (\ref{eq:H0Wald}), $R_K(\tilde{\boldsymbol{\theta}}_\gamma)$, defined in equation (\ref{def:Rao}), is given by
	\begin{equation}
		R_K(\tilde{\boldsymbol{\theta}}_\gamma) \underset{K \rightarrow \infty}{\overset{ \mathcal{L}}{\to}}  \chi_r^2. 
		\label{distrRao}
	\end{equation}
\end{theorem}
\begin{proof}
	The proof is derived from the results presented in Proposition \ref{result:rao}.
\end{proof}
The Rao test is consistent, meaning that
\begin{equation}
	\lim_{K \to \infty} \Pr\left(R_K(\tilde{\boldsymbol{\theta}}_\gamma) > \chi_{r,\alpha}^2\right) = 1,
\end{equation}
According to (\ref{distrRao}), we reject the null hypothesis stated in (\ref{eq:H0Wald}), at level $\alpha$, if
\begin{equation}
	R_K(\tilde{\boldsymbol{\theta}}_\gamma) > \chi_{r,\alpha}^2 \text{ .}
\end{equation}

\section{Numerical Analysis}

In this section, we will perform a numerical analysis to illustrate the behavior of the estimators under different conditions. We will carry out Monte Carlo simulations and use real data to evaluate and compare the efficiency and robustness of the WMDPDE estimators against the MLE. The analysis will focus on studying the errors induced by contaminating the samples, as well as calculating the empirical levels obtained from the Wald-type test statistics.

\subsection{Monte Carlo Simulation}
In this section, we perform a Monte Carlo simulation with 3000 iterations to assess the behavior of the WMDPDEs. We aim to demonstrate that the WMDPDE estimators studied in the previous sections are more robust than the MLE in the presence of outliers or atypical values.

We consider one ( \( J = 1 \) ) stress factor and \( I = 9 \) testing conditions derived from the combination of the normalized stress factor \( x \in \{0, 0.5, 1\} \) and inspection times \( \tau \in \{1, 1.5, 2.5\} \), with \( K_i = 100 \) devices for each testing condition \( i = 1, \ldots, I \), as shown in Table \ref{tabla:datos_dispositivos_modificada}.
\begin{table}[h]
	\centering
	\scriptsize
	\captionsetup{justification=centering, skip=5pt, position=above}
	\caption{Model Data with Inspection Times, Stress Factor, Failures, and Devices.}
	\begin{tabular}{|c|c|c|c|c|}
		\hline
		\textbf{Condition} & \textbf{Inspection Time (\( \tau_i \))} & \textbf{Stress Factor (\( x_i \))} & \textbf{Failures (\( n_i \))} & \textbf{Devices (\( K_i \))} \\
		\hline
		1 & 1 & 0   & $n_1$ & 100 \\
		2 & 1 & 0.5 & $n_2$ & 100 \\
		3 & 1 & 1   & $n_3$ & 100 \\
		4 & 1.5 & 0   & $n_4$ & 100 \\
		5 & 1.5 & 0.5 & $n_5$ & 100 \\
		6 & 1.5 & 1   & $n_6$ & 100 \\
		7 & 2.5 & 0   & $n_7$ & 100 \\
		8 & 2.5 & 0.5 & $n_8$ & 100 \\
		9 & 2.5 & 1   & $n_9$ & 100 \\
		\hline
	\end{tabular}
	
	\label{tabla:datos_dispositivos_modificada}
\end{table}
The failures are simulated with \( \boldsymbol{\theta} = (a_0, a_1, b_0, b_1) = (1, -0.5, 0.8, 0.4) \), meaning that failures in each testing condition are generated according to a log-logistic distribution with parameters \( \alpha_i \) and \( \beta_i \), as described in Table \ref{tabla:alpha_beta}.
\begin{table}[h]
	\centering
	\scriptsize
	\captionsetup{justification=centering, skip=5pt, position=above}
	\caption{Values of \(\alpha_i\) and \(\beta_i\) for each Testing Condition.}
	\begin{tabular}{|c|c|c|}
		\hline
		\textbf{Condition ( \(i\) )} & \(\alpha_i\) & \(\beta_i\) \\
		\hline
		1 & 2.718282 & 2.225541 \\
		2 & 2.117000 & 2.718282 \\
		3 & 1.648721 & 3.320117 \\
		4 & 2.718282 & 2.225541 \\
		5 & 2.117000 & 2.718282 \\
		6 & 1.648721 & 3.320117 \\
		7 & 2.718282 & 2.225541 \\
		8 & 2.117000 & 2.718282 \\
		9 & 1.648721 & 3.320117 \\
		\hline
	\end{tabular}
	\label{tabla:alpha_beta}
\end{table}
In order to simulate the sample of failures \( \{n_{i1}, \dots, n_{iJ}\} \) under conditions \((x_{ij}, \tau_j)\), the random variable \( X_{ij} \) is defined as:
\[
X_{ij} =
\begin{cases}
	1 & \text{if the lifetime } > \tau_j \text{ (SUCCESS)} \\
	0 & \text{if the lifetime } \leq \tau_j \text{ (FAILURE)}
\end{cases}
\sim \text{Be}(F_{\boldsymbol{\theta}}(\tau_{i}, \boldsymbol{x}_{i})).
\]
Considering independent and identically distributed random variables:
\[
N_{ij} = \sum_{t=1}^{K_i} X_{i,j}^t \sim \text{Bi}(K_i, F_{\boldsymbol{\theta}}(\tau_{i}, \boldsymbol{x}_{i})).
\]

We contaminate the sample by progressively decreasing the third parameter of \( \boldsymbol{\theta} = (1, -0.5, 0.8, 0.4) \) value when simulating failures for the first testing condition. The contamination degree applied to the parameter is calculated as the relative deviation of the contaminated parameter value \( \boldsymbol{\theta}_{\text{contaminated}} \) from the original value \( \boldsymbol{\theta} \):
\[
\text{Contamination Degree} = \frac{\lvert \boldsymbol{\theta}_{\text{contaminated}} - \boldsymbol{\theta} \rvert}{\lvert \boldsymbol{\theta} \rvert}.
\]
To evaluate and compare the estimations, the Root Mean Squared Error (RMSE) is used for each estimated parameter:
\[
\text{RMSE}_j = \sqrt{\frac{1}{2} (\hat{\boldsymbol{\theta}}_j - \boldsymbol{\theta}_j )^2 }.
\]
\begin{figure}[H]
	\centering
	\includegraphics[width=1\linewidth]{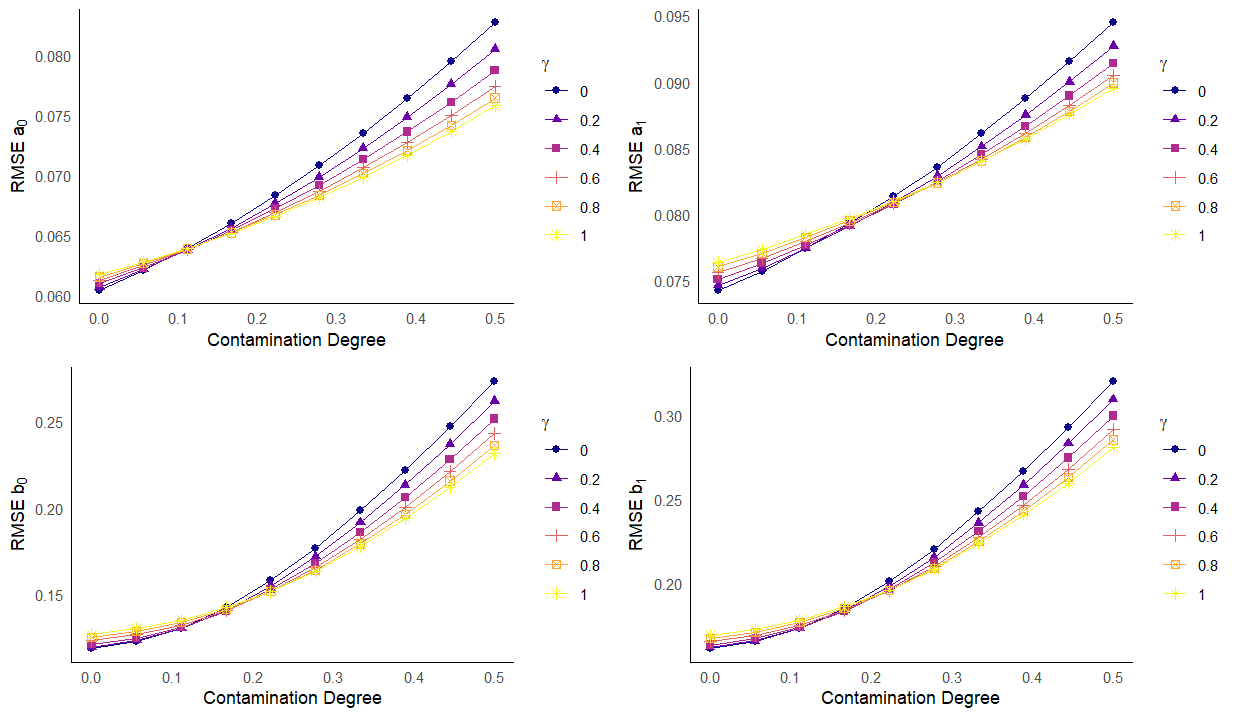}
	\caption{Root mean square errors of parameter estimates under contaminated data for different $\gamma$ values}
	\label{fig:rmse}
\end{figure}

As shown in Figure \ref{fig:rmse}, without contamination, the MLE presents lower errors. However, as the contamination degree increases, a trade-off between efficiency and robustness is observed. Based on the obtained results, we note that the MLE demonstrates high efficiency in the absence of outliers but lacks robustness when outliers are present in the data. On the other hand, the WMDPDE estimators show lower errors than the MLE in the presence of outliers while remaining efficient in their absence, because the error is not significantly higher compared to the MLE.

\subsubsection*{Robustness of Wald-type test statistics}

\begin{figure}[h]
	\centering
	\includegraphics[width=1\linewidth]{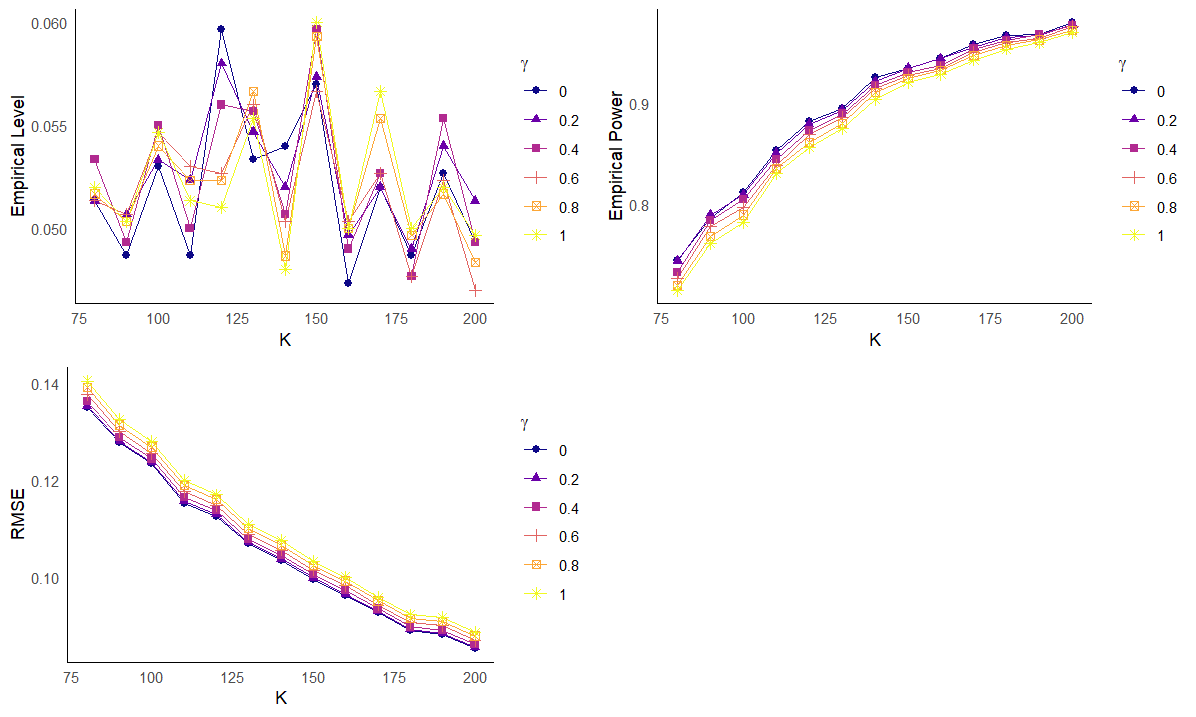}
	\caption{RMSE (bottom left panel), Empirical Level (top left) and Empirical Power (top right) under pure data for different numbers of devices (K)}
	\label{fig:dispositivos}
\end{figure}

\begin{figure}[h]
	\centering
	\includegraphics[width=1\linewidth]{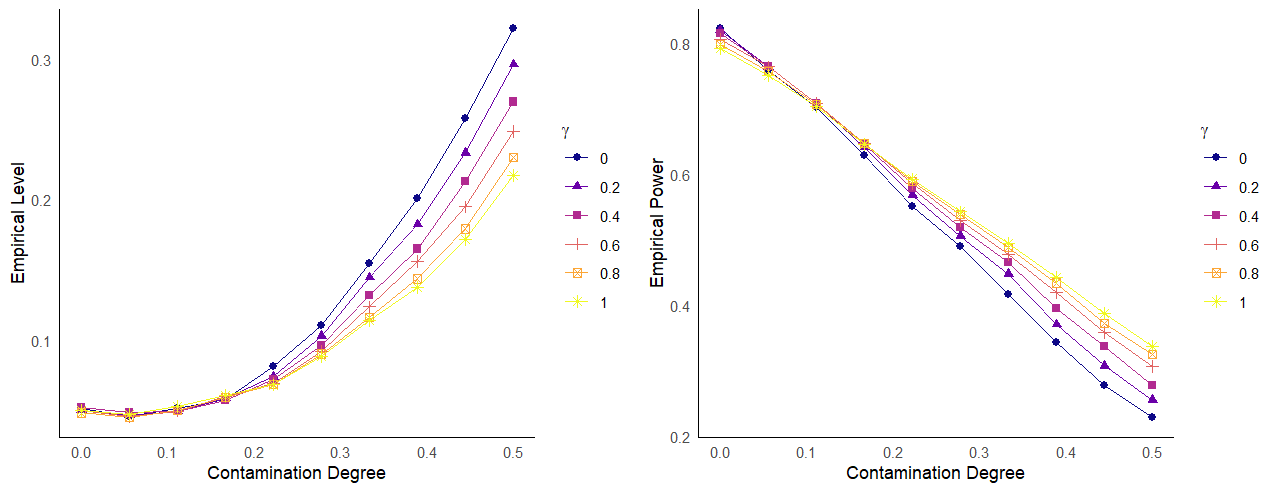}
	\caption{Wald-type tests empirical levels and empirical powers under contaminated data for different $\gamma$ values}
	\label{fig:wald}
\end{figure}
Let us now evaluate the robustness of the Wald-type tests statistic based on the previously developed WMDPDE. A Monte Carlo simulation of size 3000 is performed under the same conditions as described earlier, focusing on the empirical significance level. This level quantifies the probability of committing a Type I error, that is, rejecting the null hypothesis when it is true. The significance level is determined by the proportion of test statistics whose absolute value exceeds the critical value corresponding to a 95\% confidence level, following a \(\chi^2\) distribution with one degree of freedom. The true null hypothesis \(H_0 : b_0 = 0.8\) is tested against the alternative \(H_1 : b_0 \neq 0.8\). Contamination is introduced by perturbing this parameter in the first testing condition following the previously described procedure.

Figure \ref{fig:dispositivos} presents the mean empirical levels, empirical power and root mean square errors calculated for different values of \(\gamma\) and devices under pure data. In this case, the MLE when $\gamma=0$ presents the lowest root mean square error for each sample size, and the highest empirical power. Figure \ref{fig:wald} presents the mean empirical levels calculated for different values of \(\gamma\) and degrees of contamination. It can be observed that, for a given sample, the p-value for rejecting the null hypothesis remains similar across \(\gamma\) values (around 0.05) at low contamination levels. However, beyond 20\% contamination, the MLE (\(\gamma = 0\)) begins to diverge to higher levels. In contrast, the level remains more stable for higher values of \(\gamma\). This reveals that WMDPDEs with higher \(\gamma\) values are less sensitive to sample contamination, reducing the frequency of rejecting the null hypothesis compared to the MLE.
On the other hand, the empirical powers have been calculated under the alternative hypothesis \(H_1 : b_0 = 0.35\). It is observed that the empirical power is similar, although slightly higher for the MLE in the absence of contamination, and it progressively decreases until the curves invert, with the empirical power being higher for the WMDPDE. These results highlight the robustness of the Wald tests based on the WMDPDE compared to the one based on the MLE for one-shot devices with lifetimes under the log-logistic distribution.

\subsection{Real Data Analysis}

In this section we present a numerical example to illustrate the estimators developed. Let us focus on the real dataset from Fan (2009), an experiment that studies the behavior of 90 electroexplosive devices under different temperature levels and inspection times. Table \ref{tabla:fallos_electroexplosivos} summarizes the results obtained in the experiment for each inspection time $\tau_i = (10,20,30)$ and each temperature (stress factor) expressed in Kelvin degrees, where $x = \frac{1}{\text{Temp}} = \left(\frac{1}{308}, \frac{1}{318}, \frac{1}{328}\right)$.

\vspace{3mm}

\begin{table}[h]
	\centering
	\scriptsize
	\caption{Failures observed in electroexplosive devices under CSALTs for each temperature and inspection time.}
	\begin{tabular}{|c|c|c|c|c|}
		\hline
		\textbf{Group} & \textbf{Inspection Time} & \textbf{Temperature (K)} & \textbf{Number of Devices} & \textbf{Failures} \\ \hline
		1 & 10 & 308 & 10 & 3 \\ 
		2 & 10 & 318 & 10 & 1 \\
		3 & 10 & 328 & 10 & 6 \\ 
		4 & 20 & 308 & 10 & 3 \\ 
		5 & 20 & 318 & 10 & 7 \\ 
		6 & 20 & 328 & 10 & 7 \\ 
		7 & 30 & 308 & 10 & 7 \\ 
		8 & 30 & 318 & 10 & 7 \\ 
		9 & 30 & 328 & 10 & 9 \\  \hline
	\end{tabular}
	\label{tabla:fallos_electroexplosivos}
\end{table}

Table \ref{tabla:theta_y_media_resumida} presents the estimates obtained for $\gamma \in$ \{0, 0.1, 0.2, 0.3, 0.4, 0.5, 0.6, 0.7, 0.8, 0.9, 1\}, as well as the mean lifetimes, which are calculated as:
\begin{align*}
	\text{E[$T_{ik}$]} &=  \frac{\alpha_i \cdot \pi / \beta_i}{\sin(\pi / \beta_i)} = \frac{\left(\exp \left( \sum_{j=0}^{J} a_j x_{ij} \right) \cdot \pi\right) / \exp \left( \sum_{j=0}^{J} b_j x_{ij} \right)}{\sin \left(\pi / \exp \left( \sum_{j=0}^{J} b_j x_{ij} \right) \right)}.
\end{align*}
\begin{figure}[h]
	\centering
	\includegraphics[width=0.65\linewidth]{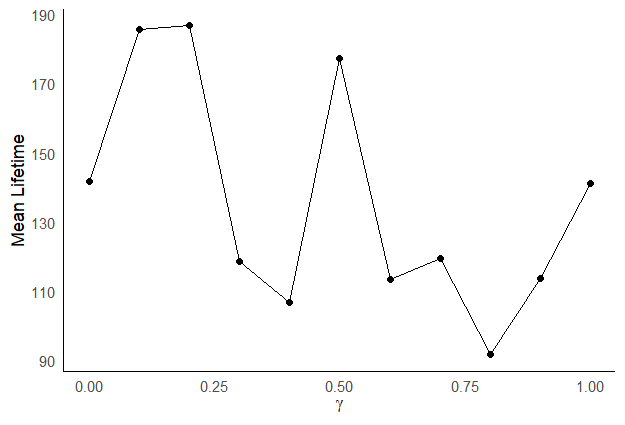}
	\caption{Estimated mean lifetime until failure for different values of $\gamma$ under normal temperature conditions ($x=\frac{1}{298}$)}
	\label{fig:est_condnormales}
\end{figure}
As expected, the estimated mean lifetime tends to decrease as the temperature increases. Additionally, when analyzing the values obtained for each temperature level, there is no clear increasing or decreasing trend based on \(\gamma\) within each level. For instance, at the lowest temperature, the maximum lifetime is achieved with \(\gamma = 0.2\), while the minimum is observed at \(\gamma = 0.8\). This behavior suggests that the estimates are not significantly influenced by outliers.
\vspace{3mm}
\begin{table}[h]
	\centering
	\scriptsize
	\caption{WMDPDE estimates and mean lifetimes for different values of $\gamma$.}
	\begin{tabular}{|c|c|c|c|c|c|c|c|}
		\hline
		\textbf{$\gamma$} & \textbf{$\hat{a_0}$} & \textbf{$\hat{a_1}$} & \textbf{$\hat{b_0}$} & \textbf{$\hat{b_1}$} & \textbf{E[$T_{x=\frac{1}{308}}$]} & \textbf{E[$T_{x=\frac{1}{318}}$]} & \textbf{E[$T_{x=\frac{1}{328}}$]} \\ \hline
		0.0 & -10.6674 & 4291.109 & 4.3174 & -1202.5600 & 62.2661 & 32.1963 & 18.3781 \\ \hline
		0.1 & -11.0716 & 4421.303 & 5.0544 & -1440.1036 & 68.5913 & 32.8659 & 17.9939 \\ \hline
		0.2 & -11.1037 & 4430.991 & 5.0508 & -1439.4333 & 68.7738 & 32.8853 & 17.9770 \\ \hline
		0.3 & -10.4249 & 4213.497 & 3.7787 & -1026.4958 & 57.6527 & 31.4461 & 18.5268 \\ \hline
		0.4 & -10.0718 & 4101.951 & 3.5032 & -935.2010  & 54.9180 & 31.0031 & 18.7098 \\ \hline
		0.5 & -11.0531 & 4415.347 & 5.0611 & -1439.0802 & 66.9270 & 32.3657 & 17.8135 \\ \hline
		0.6 & -10.7194 & 4306.893 & 3.4816 & -931.1694  & 56.7897 & 31.2977 & 18.4795 \\ \hline
		0.7 & -10.8501 & 4348.726 & 3.6965 & -999.8137  & 57.9620 & 31.3720 & 18.3071 \\ \hline
		0.8 & -9.7569  & 4000.748 & 2.9593 & -756.5985  & 50.8951 & 30.1621 & 18.8177 \\ \hline
		0.9 & -10.9808 & 4390.145 & 3.4353 & -915.5106  & 56.8226 & 31.1828 & 18.3119 \\ \hline
		1.0 & -11.2055 & 4462.832 & 4.2951 & -1192.1166 & 61.9521 & 31.7602 & 17.9254 \\ \hline
	\end{tabular}
	\label{tabla:theta_y_media_resumida}
\end{table}
\vspace{2mm}

Moreover, Figure \ref{fig:est_condnormales} illustrates the estimated mean lifetimes under normal conditions, where the temperature is 298K. The resulting mean lifetime is higher, as the temperature is lower than the stress levels considered in the experiment.
\newpage
\textbf{Goodness-of-Fit Test}
To verify whether it is appropriate to assume a log-logistic distribution for the lifetimes of the devices or not, a goodness-of-fit test based on the chi-square statistic is performed. The total sample size, observed data, and theoretical probabilities are:

\begin{itemize}
	\item Sample size:
	\(
	N = \sum_{i=1}^9 K_i = 90.
	\)
	\item Observed frequencies (\(n_i, K-n_i\)): 
	\( \{3, 3, 7, 1, 5, 7, 6, 7, 9, 7, 7, 3, 9, 5, 3, 4, 3, 1\} \).
	\item Theoretical probabilities (\(p_i, 1-p_i\)): \( 
	\{ \) 0.0210, 0.0444, 0.0613, 0.0322, 0.0635, 0.0808, 0.0496, 0.0836, 0.0965, 0.0901, 0.0667, 0.0498, 0.0789, 0.0476, 0.0303, 0.0615, 0.0275, 0.0146 \( \}.
	\)
\end{itemize}

The null hypothesis (\(H_0\)) assumes that the lifetimes follow a log-logistic distribution, where the theoretical probabilities are calculated based on  
$\boldsymbol{\theta} = (-10.6674, 4291.109, 4.3174, -1202.5600)$ corresponding to the MLE. The chi-square statistic is defined as:
\(
T = \sum_{i=1}^{18} \frac{(n_i - N \cdot p_i)^2}{N \cdot p_i}.
\)

Additionally, the degrees of freedom (\(df\)) are calculated as:
\(
df = k - (p - 1) = 18 - (4-1) = 15,
\)
where \(k = 18\) refers to the number of categories and \(p=4\) to the number of estimated parameters.

The calculated chi-square statistic is:
\[
T = \sum_{i=1}^{18} \frac{(n_i - N \cdot p_i)^2}{N \cdot p_i} = 5.297895,
\]
and the associated p-value is
\(
p = 1 - P(X \leq T) = 0.9893,
\)
where \(X\) follows a \(\chi^2\) distribution with 15 degrees of freedom.

Since \(T = 5.297895 < q_{\chi^2}(0.95, 15) = 24.995790\) and \(p = 0.9893 > 0.05\), the null hypothesis \(H_0\) is not rejected. This indicates that the observed data are consistent with the hypothesis that the underlying distribution of the lifetimes is log-logistic. Using the same procedure with the theoretical probabilities obtained from the estimations for $\gamma > 0$, the resulting statistics and p-values, all close to 1, are obtained.

\begin{table}[h]
	\centering
	\scriptsize
	\caption{Chi-Square Statistics (T) and p-values for each $\gamma$ estimation.}
	\begin{tabular}{|c|c|c|c|c|c|c|c|c|c|c|c|}
		\hline
		\textbf{$\gamma$} & 0.0 & 0.1 & 0.2 & 0.3 & 0.4 & 0.5 & 0.6 & 0.7 & 0.8 & 0.9 & 1.0 \\ 
		\hline
		\textbf{T} & 5.2979 & 5.2673 & 5.2636 & 5.3504 & 5.3936 & 5.2842 & 5.3895 & 5.3800 & 5.4680 & 5.4204 & 5.3546 \\ 
		\hline
		\textbf{p-value} & 0.9968 & 0.9969 & 0.9969 & 0.9966 & 0.9964 & 0.9968 & 0.9964 & 0.9965 & 0.9961 & 0.9963 & 0.9966 \\ 
		\hline
	\end{tabular}
	\label{tab:chi_squared_p_values}
\end{table}

\section{Concluding Remarks}
In this paper, robust estimators and test statistics for one-shot device data under log-logistic lifetime distortions have been developed. Through both theoretical and empirical analyses, the robustness of the WMDPDE has been demonstrated. The performance of the WMDPDE has been compared to the classical MLE in terms of efficiency and robustness, showing an appealing gain in robustness with a small loss of efficiency. Noteworthy, the WDPDE depends on a tuning parameter\(\gamma\) controlling the trade-off between robustness and efficiency in the estimation; the greater \(\gamma\), the more robust but less efficient the estimator will be.
Our simulation results indicate that moderate values of the tuning parameter offer a significant advantage in terms of robustness when dealing with data contamination, illustrating the practical utility of the methods.

\section*{Fundings}
This work was supported by the Spanish Grants PID2021-124933NB-I00 and FPU/018240  and Natural Sciences and Engineering Research Council of Canada (of the first author) through an Individual Discovery Grant (No. 20013416) and a Departmental Collaboration Scholarship of M. González. . M. Jaenada and L. Pardo are members of the Interdisciplinary Mathematics Institute (IMI).

\appendix
\section{Appendix: Proofs of the main results}

\subsection{Proof of Theorem \ref{th:ecuaciones}}
\label{suplmaterial_1}
We have,
\begin{equation*}
	\begin{aligned}
		\frac{\partial }{\partial \boldsymbol{\theta}} \sum_{i=1}^{I} \frac{K_i}{K} D^*_\gamma (\hat{\boldsymbol{p_i}}, \boldsymbol{\pi}_i(\boldsymbol{\theta})) = \boldsymbol{\mathbf{0}}_{2(J+1)} \text{ .}
	\end{aligned}
\end{equation*}
Thus,
\begin{equation*}
	\frac{\partial }{\partial \boldsymbol{\theta}} \sum_{i=1}^{I} \frac{K_i}{K} \left[ \left( \pi_{i1}^{\gamma+1}(\boldsymbol{\theta}) + \pi_{i2}^{\gamma+1}(\boldsymbol{\theta}) \right) - \frac{\gamma + 1}{\gamma} \left( \hat{p}_{i1} \pi_{i1}(\boldsymbol{\theta})^\gamma + \hat{p}_{i2} \pi_{i2}(\boldsymbol{\theta})^\gamma \right) \right] = \boldsymbol{\mathbf{0}}_{2(J+1)} \text{ ,}
\end{equation*}
which is equivalent to
\begin{align*}
	\frac{\partial }{\partial \boldsymbol{\theta}} \sum_{i=1}^{I} \frac{K_i}{K}  & \left[ \left( F_{\boldsymbol{\theta}}(\tau_i, \boldsymbol{x}_{i})^{\gamma+1} + (1-F_{\boldsymbol{\theta}}(\tau_i, \boldsymbol{x}_{i}) )^{\gamma+1} \right) -  \frac{\gamma + 1}{\gamma} \left( \frac{n_i}{K_i} F_{\boldsymbol{\theta}}(\tau_i, \boldsymbol{x}_{i})^\gamma + \frac{K_i - n_i}{K_i} (1- F_{\boldsymbol{\theta}}(\tau_i, \boldsymbol{x}_{i}))^\gamma \right) \right] = \boldsymbol{\mathbf{0}}_{2(J+1)}  \text{ .}
\end{align*}
That is,
\begin{align*}
	\sum_{i=1}^{I} \frac{K_i}{K} &\biggr[ (\gamma + 1) F_{\boldsymbol{\theta}}(\tau_i, \boldsymbol{x}_{i})^\gamma \frac{\partial F_{\boldsymbol{\theta}}(\tau_i, \boldsymbol{x}_{i})}{\partial \boldsymbol{\theta}} + (\gamma + 1) (1 - F_{\boldsymbol{\theta}}(\tau_i, \boldsymbol{x}_{i}))^\gamma \left(-\frac{\partial F_{\boldsymbol{\theta}}(\tau_i, \boldsymbol{x}_{i})}{\partial \boldsymbol{\theta}}\right) \\
	& - \frac{\gamma + 1}{\gamma} \left( \frac{n_i}{K_i} ( \gamma ) F_{\boldsymbol{\theta}}(\tau_i, \boldsymbol{x}_{i})^{\gamma-1} \frac{\partial F_{\boldsymbol{\theta}}(\tau_i, \boldsymbol{x}_{i})}{\partial \boldsymbol{\theta}} \right) \\
	& - \frac{\gamma + 1}{\gamma} \left( \frac{K_i - n_i}{K_i} (\gamma) (1- F_{\boldsymbol{\theta}}(\tau_i, \boldsymbol{x}_{i}))^{\gamma-1} \left(- \frac{\partial F_{\boldsymbol{\theta}}(\tau_i, \boldsymbol{x}_{i})}{\partial \boldsymbol{\theta}} \right) \right) \biggr] = \boldsymbol{\mathbf{0}}_{2(J+1)} 
\end{align*}
then
\begin{align*}
	\sum_{i=1}^{I} \frac{K_i}{K} & \left[ (\gamma + 1) F_{\boldsymbol{\theta}}(\tau_i, \boldsymbol{x}_{i})^\gamma - (\gamma + 1) (1 - F_{\boldsymbol{\theta}}(\tau_i, \boldsymbol{x}_{i}))^\gamma  - ( \gamma + 1) \left( \frac{n_i}{K_i}  F_{\boldsymbol{\theta}}(\tau_i, \boldsymbol{x}_{i})^{\gamma-1}  \right) \right. \\
	& \left.  + ( \gamma + 1) \left( \frac{K_i - n_i}{K_i}  (1- F_{\boldsymbol{\theta}}(\tau_i, \boldsymbol{x}_{i}))^{\gamma-1}   \right) \right]  \frac{\partial F_{\boldsymbol{\theta}}(\tau_i, \boldsymbol{x}_{i})}{\partial \boldsymbol{\theta}} = \boldsymbol{\mathbf{0}}_{2(J+1)}  \text{ ,}
\end{align*}
and therefore
\begin{align*}
	\sum_{i=1}^{I} K_i & \left\{  \biggr[ F_{\boldsymbol{\theta}}(\tau_i, \boldsymbol{x}_{i})^\gamma  - (1 - F_{\boldsymbol{\theta}}(\tau_i, \boldsymbol{x}_{i}))^\gamma  \right. \left.  - \frac{n_i}{K_i} F_{\boldsymbol{\theta}}(\tau_i, \boldsymbol{x}_{i})^{\gamma-1} + \frac{K_i - n_i}{K_i} (1- F_{\boldsymbol{\theta}}(\tau_i, \boldsymbol{x}_{i}))^{\gamma-1}  \biggr] \frac{\partial F_{\boldsymbol{\theta}}(\tau_i, \boldsymbol{x}_{i})}{\partial \boldsymbol{\theta}} \right\} = \boldsymbol{\mathbf{0}}_{2(J+1)}  \text{ .}
\end{align*}
Simplifying, we arrive at
\begin{align*}
	\sum_{i=1}^{I} K_i & \left\{  \biggr[ F_{\boldsymbol{\theta}}(\tau_i, \boldsymbol{x}_{i})^{\gamma-1} \biggr( F_{\boldsymbol{\theta}}(\tau_i, \boldsymbol{x}_{i})  - \frac{n_i}{K_i} \biggr) + \right.  \left. (1 - F_{\boldsymbol{\theta}}(\tau_i, \boldsymbol{x}_{i}))^{\gamma-1} \biggr( - 1 + F_{\boldsymbol{\theta}}(\tau_i, \boldsymbol{x}_{i}) + \frac{K_i - n_i}{K_i} \biggr)   \biggr] \frac{\partial F_{\boldsymbol{\theta}}(\tau_i, \boldsymbol{x}_{i})}{\partial \boldsymbol{\theta}} \right\} = \boldsymbol{\mathbf{0}}_{2(J+1)} 
\end{align*}
or what is equivalent,
\begin{align*}
	\sum_{i=1}^{I} K_i & \left\{  \biggr[ \biggr( F_{\boldsymbol{\theta}}(\tau_i, \boldsymbol{x}_{i})  - \frac{n_i}{K_i} \biggr) \biggr( F_{\boldsymbol{\theta}}(\tau_i, \boldsymbol{x}_{i})^{\gamma-1}  +  (1 - F_{\boldsymbol{\theta}}(\tau_i, \boldsymbol{x}_{i}))^{\gamma-1} \biggr) \biggr] \frac{\partial F_{\boldsymbol{\theta}}(\tau_i, \boldsymbol{x}_{i})}{\partial \boldsymbol{\theta}} \right\} = \boldsymbol{\mathbf{0}}_{2(J+1)}  \text{ .}
\end{align*} 
Based on this result, and by replacing $\frac{\partial F_{\alpha_i, \beta_i}(\tau_i, \boldsymbol{x}_i)}{\partial a_j}$ from equation (\ref{eq:derFaj}), we obtain:
\begin{align*}
	\sum_{i=1}^{I} K_i & \left\{ \left[ \left( \frac{\tau_i^{\beta_i}}{\tau_i^{\beta_i} + \alpha_i^{\beta_i}} - \frac{n_i}{K_i} \right)  \left( \frac{\tau_i^{\beta_i (\gamma-1)}}{ (\tau_i^{\beta_i} + \alpha_i^{\beta_i} )^{(\gamma-1)}}+  \frac{\alpha_i^{\beta_i (\gamma-1)}}{ (\tau_i^{\beta_i} + \alpha_i^{\beta_i} )^{(\gamma-1)}} \right)  \right] \left( -\frac{\beta_i \alpha_i^{\beta_i} \tau_i^{\beta_i}}{(\tau_i^{\beta_i} + \alpha_i^{\beta_i})^2} \right) x_{ij} \right\} = \boldsymbol{\mathbf{0}}_{(J+1)}
\end{align*}
then,
\begin{align*}
	- \sum_{i=1}^{I} & \left\{ \left[ \left( K_i \frac{\tau_i^{\beta_i}}{\tau_i^{\beta_i} + \alpha_i^{\beta_i}} - n_i \right) \left( \frac{\tau_i^{\beta_i (\gamma-1)} \alpha_i^{\beta_i} \tau_i^{\beta_i} + \alpha_i^{\beta_i (\gamma-1)} \alpha_i^{\beta_i} \tau_i^{\beta_i} }{ (\tau_i^{\beta_i} + \alpha_i^{\beta_i} )^{(\gamma-1)}}\right) \right] \left( \frac{\beta_i }{(\tau_i^{\beta_i} + \alpha_i^{\beta_i})^2} \right) x_{ij} \right\} = \boldsymbol{\mathbf{0}}_{(J+1)}
\end{align*}
which is equivalent to
\begin{align}
	- \sum_{i=1}^{I} \left( K_i \frac{\tau_i^{\beta_i}}{\tau_i^{\beta_i} + \alpha_i^{\beta_i}} - n_i \right) \left(  \frac{\tau_i^{\beta_i \gamma} \alpha_i^{\beta_i} + \alpha_i^{\beta_i \gamma} \tau_i^{\beta_i}}{   ( \tau_i^{\beta_i} + \alpha_i^{\beta_i} )^{\gamma +1} }  \right) \beta_i \cdot x_{ij} =  \boldsymbol{\mathbf{0}}_{(J+1)} \text{ .}
	\label{eq:1}
\end{align}

Similarly, by substituting $\frac{\partial F_{\alpha_i, \beta_i}(\tau_i, \boldsymbol{x}_i)}{\partial b_j}$ from equation (\ref{eq:derFbj}), we obtain:
\begin{align*}
	\sum_{i=1}^{I} K_i & \left\{ \left[ \left( \frac{\tau_i^{\beta_i}}{\tau_i^{\beta_i} + \alpha_i^{\beta_i}} - \frac{n_i}{K_i} \right)  \left( \frac{\tau_i^{\beta_i (\gamma-1)}}{ (\tau_i^{\beta_i} + \alpha_i^{\beta_i} )^{(\gamma-1)}}+  \frac{\alpha_i^{\beta_i (\gamma-1)}}{ (\tau_i^{\beta_i} + \alpha_i^{\beta_i} )^{(\gamma-1)}} \right)  \right] \left( \frac{\beta_i \log{\frac{\tau_i}{\alpha_i}} \alpha_i^{\beta_i} \tau_i^{\beta_i}}{(\tau_i^{\beta_i} + \alpha_i^{\beta_i})^2} \right) x_{ij} \right\} = \boldsymbol{\mathbf{0}}_{(J+1)}
\end{align*}
that is,
\begin{align}
	\sum_{i=1}^{I} \left( K_i \frac{\tau_i^{\beta_i}}{\tau_i^{\beta_i} + \alpha_i^{\beta_i}} - n_i \right) \left(  \frac{\tau_i^{\beta_i \gamma} \alpha_i^{\beta_i} + \alpha_i^{\beta_i \gamma} \tau_i^{\beta_i}}{   ( \tau_i^{\beta_i} + \alpha_i^{\beta_i} )^{\gamma +1} }  \right) \log{\frac{\tau_i}{\alpha_i}} \cdot \beta_i \cdot x_{ij} =  \boldsymbol{\mathbf{0}}_{(J+1)} \text{ .}
	\label{eq:2}
\end{align}

From expressions (\ref{eq:1}) and (\ref{eq:2}), we can represent the equations in matrix form as
\[
\sum_{i=1}^{I} \left( -\beta_i \boldsymbol{x}_i, \log \frac{\tau_i}{\alpha_i} \beta_i \boldsymbol{x}_i \right) \left( K_i \frac{\tau_i^{\beta_i}}{\tau_i^{\beta_i} + \alpha_i^{\beta_i}} - n_i \right) \left( \frac{\tau_i^{\beta_i \gamma} \alpha_i^{\beta_i} + \alpha_i^{\beta_i \gamma} \tau_i^{\beta_i}}{(\tau_i^{\beta_i} + \alpha_i^{\beta_i})^{\gamma + 1}} \right) = \boldsymbol{\mathbf{0}}_{2(J+1)}  \text{ .}
\]
\subsection{Proof of Proposition \ref{prop:eclibro}}
\label{suplmaterial_2}
The result presented in the cited paper of Balakrishnan et al (2022) established that:
\[ \sqrt{K} (\hat{\boldsymbol{\theta}}_\gamma - \boldsymbol{\theta}_0) \underset{K \rightarrow \infty}{\overset{\mathcal{L}}{\to}} \mathcal{N}( \boldsymbol{0}_{2(J+1)}, \boldsymbol{J}_\gamma^{-1}(\boldsymbol{\theta}_0) \boldsymbol{K}_\gamma (\boldsymbol{\theta}_0) \boldsymbol{J}_\gamma^{-1}(\boldsymbol{\theta}_0)),
\]    
where
\[
\boldsymbol{J}_\gamma (\boldsymbol{\theta}) = \sum_{i=1}^{I} \sum_{j=1}^{2} \frac{K_i}{K} \boldsymbol{u}_{ij} (\boldsymbol{\theta}) \boldsymbol{u}_{ij}^T (\boldsymbol{\theta}) \pi_{ij}^{\gamma + 1} (\boldsymbol{\theta}),
\]
\[
\boldsymbol{K}_{\gamma}(\boldsymbol{\theta}) = \sum_{i=1}^{I} \frac{K_i}{K} \sum_{j=1}^2 u_{ij}(\boldsymbol{\theta}) u_{ij}(\boldsymbol{\theta})^T \pi_{ij}(\boldsymbol{\theta})^{2\gamma+1} - \sum_{i=1}^{I} \frac{K_i}{K} \xi_{i,\gamma}(\boldsymbol{\theta}) \xi_{i,\gamma}(\boldsymbol{\theta})^T 
\]
with \[ \boldsymbol{u}_{ij} (\boldsymbol{\theta}) = \frac{\partial \log \pi_{ij} (\boldsymbol{\theta})}{\partial \boldsymbol{\theta}} \textit{   and   }
\xi_{i,\gamma}(\boldsymbol{\theta}) =  \sum_{j=1}^2 u_{ij}(\boldsymbol{\theta}) \pi_{ij}(\boldsymbol{\theta})^{\gamma+1} . 
\]
We have
\[
\boldsymbol{u}_{i1} (\boldsymbol{\theta}) = \frac{\partial \log \pi_{i1} (\boldsymbol{\theta})}{\partial \boldsymbol{\theta}} = \frac{1}{F_{\boldsymbol{\theta}}(\tau_i, \boldsymbol{x}_i)} \frac{\partial F_{\boldsymbol{\theta}}(\tau_i, \boldsymbol{x}_i)}{\partial \boldsymbol{\theta}}
\]
\[
\boldsymbol{u}_{i2} (\boldsymbol{\theta}) = \frac{\partial \log \pi_{i2} (\boldsymbol{\theta})}{\partial \boldsymbol{\theta}} = - \frac{1}{R_{\boldsymbol{\theta}}(\tau_i, \boldsymbol{x}_i)} \frac{\partial F_{\boldsymbol{\theta}}(\tau_i, \boldsymbol{x}_i)}{\partial \boldsymbol{\theta}}.
\]
On the other hand,
\begin{align*}
	\boldsymbol{u}_{i1} (\boldsymbol{\theta}) \boldsymbol{u}_{i1} (\boldsymbol{\theta})^T \pi_{i1} (\boldsymbol{\theta})^{\gamma +1} &= \frac{1}{F_{\boldsymbol{\theta}}(\tau_i, \boldsymbol{x}_i)^2} \frac{\partial F_{\boldsymbol{\theta}}(\tau_i, \boldsymbol{x}_i)}{\partial \boldsymbol{\theta}} \frac{\partial F_{\boldsymbol{\theta}}(\tau_i, \boldsymbol{x}_i)}{\partial \boldsymbol{\theta}^T} F_{\boldsymbol{\theta}}(\tau_i, \boldsymbol{x}_i)^{\gamma+1}  \frac{\partial F_{\boldsymbol{\theta}}(\tau_i, \boldsymbol{x}_i)}{\partial \boldsymbol{\theta}} \frac{\partial F_{\boldsymbol{\theta}}(\tau_i, \boldsymbol{x}_i)}{\partial \boldsymbol{\theta}^T} F_{\boldsymbol{\theta}}(\tau_i, \boldsymbol{x}_i)^{\gamma-1}  
\end{align*}
and
\begin{align*}
	\boldsymbol{u}_{i2} (\boldsymbol{\theta}) \boldsymbol{u}_{i2} (\boldsymbol{\theta})^T \pi_{i2} (\boldsymbol{\theta})^{\gamma +1} &= \frac{1}{R_{\boldsymbol{\theta}}(\tau_i, \boldsymbol{x}_i)^2} \frac{\partial F_{\boldsymbol{\theta}}(\tau_i, \boldsymbol{x}_i)}{\partial \boldsymbol{\theta}} \frac{\partial F_{\boldsymbol{\theta}}(\tau_i, \boldsymbol{x}_i)}{\partial \boldsymbol{\theta}^T} R_{\boldsymbol{\theta}}(\tau_i, \boldsymbol{x}_i)^{\gamma+1}  \frac{\partial F_{\boldsymbol{\theta}}(\tau_i, \boldsymbol{x}_i)}{\partial \boldsymbol{\theta}} \frac{\partial F_{\boldsymbol{\theta}}(\tau_i, \boldsymbol{x}_i)}{\partial \boldsymbol{\theta}^T} R_{\boldsymbol{\theta}}(\tau_i, \boldsymbol{x}_i)^{\gamma-1}. 
\end{align*}
Thus,
\[
\boldsymbol{J}_\gamma (\boldsymbol{\theta}) = \sum_{i=1}^{I} \frac{K_i}{K} \biggr( F_{\boldsymbol{\theta}}(\tau_i, \boldsymbol{x}_i)^{\gamma-1} + R_{\boldsymbol{\theta}}(\tau_i, \boldsymbol{x}_i)^{\gamma-1} \biggr) \frac{ \partial F_{\boldsymbol{\theta}}(\tau_i, \boldsymbol{x}_i)}{\partial \boldsymbol{\theta}} \frac{ \partial F_{\boldsymbol{\theta}}(\tau_i, \boldsymbol{x}_i)}{\partial \boldsymbol{\theta}^T}.
\]
For \(K_\gamma(\boldsymbol{\theta})\), we have
\begin{align*}  
	\xi_{i,\gamma}(\boldsymbol{\theta}) &= \sum_{j=1}^2 u_{ij}(\boldsymbol{\theta}) \pi_{ij}(\boldsymbol{\theta})^{\gamma+1}  \\
	&= u_{i1}(\boldsymbol{\theta}) \pi_{i1}(\boldsymbol{\theta})^{\gamma+1} + u_{i2}(\boldsymbol{\theta}) \pi_{i2}(\boldsymbol{\theta})^{\gamma+1} \\
	&= \frac{\partial F_{\boldsymbol{\theta}}(\tau_i, \boldsymbol{x}_i)}{\partial \boldsymbol{\theta}} F_{\boldsymbol{\theta}}(\tau_i, \boldsymbol{x}_i)^{\gamma} - \frac{\partial F_{\boldsymbol{\theta}}(\tau_i, \boldsymbol{x}_i)}{\partial \boldsymbol{\theta}} R_{\boldsymbol{\theta}}(\tau_i, \boldsymbol{x}_i)^{\gamma}.
\end{align*}
and
\begin{align*}  
	\xi_{i,\gamma}(\boldsymbol{\theta}) \xi_{i,\gamma}(\boldsymbol{\theta})^T 
	&= \biggr[ \frac{\partial F_{\boldsymbol{\theta}}(\tau_i, \boldsymbol{x}_i)}{\partial \boldsymbol{\theta} } F_{\boldsymbol{\theta}}(\tau_i, \boldsymbol{x}_i)^{\gamma} - \frac{\partial F_{\boldsymbol{\theta}}(\tau_i, \boldsymbol{x}_i)}{\partial \boldsymbol{\theta} } R_{\boldsymbol{\theta}}(\tau_i, \boldsymbol{x}_i)^{\gamma} \biggr] \\
	&\quad \cdot \biggr[ \frac{\partial F_{\boldsymbol{\theta}}(\tau_i, \boldsymbol{x}_i)}{\partial \boldsymbol{\theta}^T } F_{\boldsymbol{\theta}}(\tau_i, \boldsymbol{x}_i)^{\gamma} - \frac{\partial F_{\boldsymbol{\theta}}(\tau_i, \boldsymbol{x}_i)}{\partial \boldsymbol{\theta}^T } R_{\boldsymbol{\theta}}(\tau_i, \boldsymbol{x}_i)^{\gamma} \biggr] \\
	&= \frac{\partial F_{\boldsymbol{\theta}}(\tau_i, \boldsymbol{x}_i)}{\partial \boldsymbol{\theta} } \frac{\partial F_{\boldsymbol{\theta}}(\tau_i, \boldsymbol{x}_i)}{\partial \boldsymbol{\theta}^T } \biggr[ F_{\boldsymbol{\theta}}(\tau_i, \boldsymbol{x}_i)^{2\gamma} + R_{\boldsymbol{\theta}}(\tau_i, \boldsymbol{x}_i)^{2\gamma} - 2 F_{\boldsymbol{\theta}}(\tau_i, \boldsymbol{x}_i)^\gamma R_{\boldsymbol{\theta}}(\tau_i, \boldsymbol{x}_i)^\gamma \biggr].
\end{align*}
Substituting into \(K_\gamma(\boldsymbol{\theta})\), we get
\allowdisplaybreaks
\begin{align*}
	\boldsymbol{K}_{\gamma}(\boldsymbol{\theta}) &=
	\sum_{i=1}^{I} \frac{K_i}{K}
	\frac{\partial F_{\boldsymbol{\theta}}(\tau_i, \boldsymbol{x}_i)}{\partial \boldsymbol{\theta} } \frac{\partial F_{\boldsymbol{\theta}}(\tau_i, \boldsymbol{x}_i)}{\partial \boldsymbol{\theta}^T } \biggr[ F_{\boldsymbol{\theta}}(\tau_i, \boldsymbol{x}_i)^{2\gamma-1} + R_{\boldsymbol{\theta}}(\tau_i, \boldsymbol{x}_i)^{2\gamma-1} \\
	&\quad - F_{\boldsymbol{\theta}}(\tau_i, \boldsymbol{x}_i)^{2\gamma} - R_{\boldsymbol{\theta}}(\tau_i, \boldsymbol{x}_i)^{2\gamma} + 2 F_{\boldsymbol{\theta}}(\tau_i, \boldsymbol{x}_i) R_{\boldsymbol{\theta}}(\tau_i, \boldsymbol{x}_i) \biggr] \\
	&= \sum_{i=1}^{I} \frac{K_i}{K}
	\frac{\partial F_{\boldsymbol{\theta}}(\tau_i, \boldsymbol{x}_i)}{\partial \boldsymbol{\theta} } \frac{\partial F_{\boldsymbol{\theta}}(\tau_i, \boldsymbol{x}_i)}{\partial \boldsymbol{\theta}^T } \biggr[ F_{\boldsymbol{\theta}}(\tau_i, \boldsymbol{x}_i)^{2\gamma-1}R_{\boldsymbol{\theta}}(\tau_i, \boldsymbol{x}_i) \\
	&\quad + R_{\boldsymbol{\theta}}(\tau_i, \boldsymbol{x}_i)^{2\gamma-1}F_{\boldsymbol{\theta}}(\tau_i, \boldsymbol{x}_i) - 2 F_{\boldsymbol{\theta}}(\tau_i, \boldsymbol{x}_i)^\gamma R_{\boldsymbol{\theta}}(\tau_i, \boldsymbol{x}_i)^\gamma \biggr] \\
	&= \sum_{i=1}^{I} \frac{K_i}{K}
	\frac{\partial F_{\boldsymbol{\theta}}(\tau_i, \boldsymbol{x}_i)}{\partial \boldsymbol{\theta} } \frac{\partial F_{\boldsymbol{\theta}}(\tau_i, \boldsymbol{x}_i)}{\partial \boldsymbol{\theta}^T } F_{\boldsymbol{\theta}}(\tau_i, \boldsymbol{x}_i) R_{\boldsymbol{\theta}}(\tau_i, \boldsymbol{x}_i) \biggr[ F_{\boldsymbol{\theta}}(\tau_i, \boldsymbol{x}_i)^{2\gamma-2} \\
	&\quad + R_{\boldsymbol{\theta}}(\tau_i, \boldsymbol{x}_i)^{2\gamma-2} - 2 F_{\boldsymbol{\theta}}(\tau_i, \boldsymbol{x}_i)^{\gamma-1} R_{\boldsymbol{\theta}}(\tau_i, \boldsymbol{x}_i)^{\gamma-1} \biggr] \\
	&= \sum_{i=1}^{I} \frac{K_i}{K} F_{\boldsymbol{\theta}}(\tau_i, \boldsymbol{x}_i) R_{\boldsymbol{\theta}}(\tau_i, \boldsymbol{x}_i) \biggr( F_{\boldsymbol{\theta}}(\tau_i, \boldsymbol{x}_i)^{\gamma-1} + R_{\boldsymbol{\theta}}(\tau_i, \boldsymbol{x}_i)^{\gamma-1} \biggr)^2 \frac{ \partial F_{\boldsymbol{\theta}}(\tau_i, \boldsymbol{x}_i)}{\partial \boldsymbol{\theta} } \frac{ \partial F_{\boldsymbol{\theta}}(\tau_i, \boldsymbol{x}_i)}{\partial \boldsymbol{\theta}^T}.
\end{align*}

\subsection{Proof of Theorem \ref{th:asim_distr}}
\label{sec:th:asim_distr}
The result presented in Proposition \ref{prop:libro} establishes that:
\allowdisplaybreaks
\begin{equation*}
	\sqrt{K} (\hat{\boldsymbol{\theta}}_\gamma - \boldsymbol{\theta}_0) \underset{K \rightarrow \infty}{\overset{\mathcal{L}}{\to}} N\biggr(\mathbf{0}_{2(J+1)}, \boldsymbol{J}_\gamma^{-1}(\boldsymbol{\theta}_0) \boldsymbol{K}_\gamma (\boldsymbol{\theta}_0) \boldsymbol{J}_\gamma^{-1}(\boldsymbol{\theta}_0)\biggr)
\end{equation*}
with
\begin{equation*}
	\boldsymbol{J}_\gamma (\boldsymbol{\theta}) = \sum_{i=1}^{I}   \frac{K_i}{K}  \biggr( F_{\boldsymbol{\theta}}(\tau_i, \boldsymbol{x}_i)^{\gamma-1} + R_{\boldsymbol{\theta}}(\tau_i, \boldsymbol{x}_i)^{\gamma-1} \biggr) \frac{ \partial F_{\boldsymbol{\theta}}(\tau_i, \boldsymbol{x}_i)}{\partial \boldsymbol{\theta} } \frac{ \partial F_{\boldsymbol{\theta}}(\tau_i, \boldsymbol{x}_i)}{\partial \boldsymbol{\theta}^T},
\end{equation*}
\begin{equation*}
	\boldsymbol{K}_{\gamma}(\boldsymbol{\theta}) = \sum_{i=1}^{I}   \frac{K_i}{K} F_{\boldsymbol{\theta}}(\tau_i, \boldsymbol{x}_i) R_{\boldsymbol{\theta}}(\tau_i, \boldsymbol{x}_i)  \biggr( F_{\boldsymbol{\theta}}(\tau_i, \boldsymbol{x}_i)^{\gamma-1} + R_{\boldsymbol{\theta}}(\tau_i, \boldsymbol{x}_i)^{\gamma-1} \biggr)^2 \frac{ \partial F_{\boldsymbol{\theta}}(\tau_i, \boldsymbol{x}_i)}{\partial \boldsymbol{\theta} } \frac{ \partial F_{\boldsymbol{\theta}}(\tau_i, \boldsymbol{x}_i)}{\partial \boldsymbol{\theta}^T}.
\end{equation*}
In the context of a log-logistic lifetime distribution, we may substitute \( F_{\boldsymbol{\theta}}(\tau_i, x_{ij}) \) with \( F_{\alpha_i, \beta_i}(\tau_i, \boldsymbol{x}_i) \) and \( R_{\boldsymbol{\theta}}(\tau_i, x_{ij}) \) with \( R_{\alpha_i, \beta_i}(\tau_i, \boldsymbol{x}_i) \). By differentiating \( F_{\alpha_i, \beta_i}(\tau_i, \boldsymbol{x}_i) \) with respect to \( \boldsymbol{\theta} = (a_j, b_j) \), we obtain the following partial derivatives:
\begin{equation}
	\frac{\partial F_{\alpha_i, \beta_i}(\tau_i, \boldsymbol{x}_i)}{\partial a_j} = \frac{\partial F_{\alpha_i, \beta_i}(\tau_i, \boldsymbol{x}_i)}{\partial \alpha_i} \cdot \frac{\partial \alpha_i}{ \partial a_j} = \left( -\frac{\beta_i \alpha_i^{\beta_i - 1} \tau_i^{\beta_i}}{(\tau_i^{\beta_i} + \alpha_i^{\beta_i})^2} \right) \cdot \alpha_i x_{ij} = -\frac{ \alpha_i^{\beta_i} \tau_i^{\beta_i}}{(\tau_i^{\beta_i} + \alpha_i^{\beta_i})^2}  \cdot \beta_i x_{ij} \text{ ,}
	\label{eq:derFaj}
\end{equation}
and
\begin{equation}
	\frac{\partial F_{\alpha_i, \beta_i}(\tau_i, \boldsymbol{x}_i)}{\partial b_j} = \frac{\partial F_{\alpha_i, \beta_i}(\tau_i, \boldsymbol{x}_i)}{\partial \beta_i} \cdot \frac{\partial \beta_i}{ \partial b_j} = \frac{\tau_i^{\beta_i} \alpha_i^{\beta_i} \log \tau_i - \tau_i^{\beta_i} \alpha_i^{\beta_i} \log \alpha_i}{(\tau_i^{\beta_i} + \alpha_i^{\beta_i})^2} \cdot \beta_i x_{ij}  = \frac{\tau_i^{\beta_i} \alpha_i^{\beta_i} \log \frac{\tau_i}{\alpha_i} }{(\tau_i^{\beta_i} + \alpha_i^{\beta_i})^2} \cdot \beta_i x_{ij} \text{ .}
	\label{eq:derFbj}
\end{equation}
From the above, we obtain:
\begin{align*}
	\frac{ \partial F_{\boldsymbol{\theta}}(\tau_i, \boldsymbol{x}_i)}{\partial \boldsymbol{\theta} } \frac{ \partial F_{\boldsymbol{\theta}}(\tau_i, \boldsymbol{x}_i)}{\partial \boldsymbol{\theta}^T} &= \begin{pmatrix}
		-\frac{ \alpha_i^{\beta_i} \tau_i^{\beta_i}}{(\tau_i^{\beta_i} + \alpha_i^{\beta_i})^2} \cdot \beta_i x_{ij} \text{,} & \frac{\tau_i^{\beta_i} \alpha_i^{\beta_i} \log \frac{\tau_i}{\alpha_i}}{(\tau_i^{\beta_i} + \alpha_i^{\beta_i})^2} \cdot \beta_i x_{ij}
	\end{pmatrix}  \cdot \begin{pmatrix}
		-\frac{ \alpha_i^{\beta_i} \tau_i^{\beta_i}}{(\tau_i^{\beta_i} + \alpha_i^{\beta_i})^2} \cdot \beta_i x_{ij} \\ \frac{\tau_i^{\beta_i} \alpha_i^{\beta_i} \log \frac{\tau_i}{\alpha_i}}{(\tau_i^{\beta_i} + \alpha_i^{\beta_i})^2} \cdot \beta_i x_{ij}
	\end{pmatrix} \\
	&= \begin{pmatrix}
		\left( \frac{\alpha_i^{2\beta_i} \tau_i^{2\beta_i}}{(\tau_i^{\beta_i} + \alpha_i^{\beta_i})^4} \cdot \beta_i^2 x_{ij}^2 \right) & \left( -\frac{\alpha_i^{2\beta_i} \tau_i^{2\beta_i} \log \frac{\tau_i}{\alpha_i}}{(\tau_i^{\beta_i} + \alpha_i^{\beta_i})^4} \cdot \beta_i^2 x_{ij}^2 \right) \\
		\left( -\frac{\alpha_i^{2\beta_i} \tau_i^{2\beta_i} \log \frac{\tau_i}{\alpha_i}}{(\tau_i^{\beta_i} + \alpha_i^{\beta_i})^4} \cdot \beta_i^2 x_{ij}^2 \right) & \left( \frac{\tau_i^{2\beta_i} \alpha_i^{2\beta_i} (\log \frac{\tau_i}{\alpha_i})^2}{(\tau_i^{\beta_i} + \alpha_i^{\beta_i})^4} \cdot \beta_i^2 x_{ij}^2 \right) \end{pmatrix} \\
	&= \frac{\alpha_i^{2\beta_i}}{(\tau_i^{\beta_i} + \alpha_i^{\beta_i})^2} \cdot \frac{\tau_i^{2\beta_i}}{(\tau_i^{\beta_i} + \alpha_i^{\beta_i})^2} \cdot \begin{pmatrix}
		\beta_i^2 x_{ij}^2 & - \log \frac{\tau_i}{\alpha_i} \cdot \beta_i^2 x_{ij}^2 \\
		- \log \frac{\tau_i}{\alpha_i} \cdot \beta_i^2 x_{ij}^2 &  (\log \frac{\tau_i}{\alpha_i} \cdot \beta_i)^2 x_{ij}^2 \end{pmatrix}  \\
	&=  F_{\alpha_i, \beta_i}(\tau_i, \boldsymbol{x_i})^2 \cdot R_{\alpha_i, \beta_i}(\tau_i, \boldsymbol{x_i})^2  \cdot \begin{pmatrix}
		\beta_i^2 x_{ij}^2 & - \log \frac{\tau_i}{\alpha_i} \cdot \beta_i^2 x_{ij}^2 \\
		- \log \frac{\tau_i}{\alpha_i} \cdot \beta_i^2 x_{ij}^2 &  (\log \frac{\tau_i}{\alpha_i} \cdot \beta_i)^2 x_{ij}^2 \end{pmatrix} \\
	&= F_{\alpha_i, \beta_i}(\tau_i, \boldsymbol{x_i})^2 \cdot R_{\alpha_i, \beta_i}(\tau_i, \boldsymbol{x_i})^2  \cdot \boldsymbol{M}_i
\end{align*}
where
\[
\boldsymbol{M}_i = \begin{pmatrix}
	\beta_i^2 \boldsymbol{x_i} \boldsymbol{x_i}^T & -\log \frac{\tau_i}{\alpha_i} \beta_i^2 \boldsymbol{x_i} \boldsymbol{x_i}^T \\
	-\log \frac{\tau_i}{\alpha_i} \beta_i^2 \boldsymbol{x_i} \boldsymbol{x_i}^T & \left( \log \frac{\tau_i}{\alpha_i} \beta_i \right)^2  \boldsymbol{x_i} \boldsymbol{x_i}^T
\end{pmatrix} \text{ .}
\]
Thus,
\[
\boldsymbol{J}_\gamma (\boldsymbol{\theta}) = \sum_{i=1}^{I} \frac{K_i}{K} \boldsymbol{M}_i \left( R_{\alpha_i, \beta_i} (\tau_i, x_i) F_{\alpha_i, \beta_i} (\tau_i, x_i) \right)^2 \left( F_{\alpha_i, \beta_i} (\tau_i, x_i)^{\gamma - 1} + R (\tau_i, x_i)^{\gamma - 1} \right) 
\]
and
\[
\boldsymbol{K}_\gamma(\boldsymbol{\theta}) = \sum_{i=1}^{I} \frac{K_i}{K} \left( F_{\alpha_i, \beta_i}(\tau_i, \boldsymbol{x_i})^{\gamma-1} + R_{\alpha_i, \beta_i}(\tau_i, \boldsymbol{x_i})^{\gamma-1} \right)^2 F_{\alpha_i, \beta_i}(\tau_i, \boldsymbol{x_i})^3 \cdot R_{\alpha_i, \beta_i}(\tau_i, \boldsymbol{x_i})^3  \cdot \boldsymbol{M}_i.
\]

\subsection{Proof of Proposition \ref{result:rao}}
\label{sec:proof:result:rao}
First, since \(\mathbb{E}[n_i] = K_i \frac{\tau_i^{\beta_i}}{\tau_i^{\beta_i} + \alpha_i^{\beta_i}}\),
\[
\mathbb{E}[\boldsymbol{U}_{\gamma}(\boldsymbol{\theta})] =  \sum_{i=1}^I \left( -\beta_i \boldsymbol{x}_i, \log \frac{\tau_i}{\alpha_i} \beta_i \boldsymbol{x}_i \right) \left( K_i \frac{\tau_i^{\beta_i}}{\tau_i^{\beta_i} + \alpha_i^{\beta_i}} - K_i \frac{\tau_i^{\beta_i}}{\tau_i^{\beta_i} + \alpha_i^{\beta_i}} \right) \left( \frac{\tau_i^{\beta_i \gamma} \alpha_i^{\beta_i} + \alpha_i^{\beta_i \gamma} \tau_i^{\beta_i}}{(\tau_i^{\beta_i} + \alpha_i^{\beta_i})^{\gamma + 1}} \right) = \boldsymbol{0}.
\]
Regarding the variance, we have
\begin{align*}
	\text{Var}(\boldsymbol{U}_{\gamma}(\boldsymbol{\theta})) &=  \sum_{i=1}^I \left( -\beta_i \boldsymbol{x}_i, \log \frac{\tau_i}{\alpha_i} \beta_i \boldsymbol{x}_i \right)^2 \text{Var}\left( K_i \frac{\tau_i^{\beta_i}}{\tau_i^{\beta_i} + \alpha_i^{\beta_i}} - n_i \right) \left( \frac{\tau_i^{\beta_i \gamma} \alpha_i^{\beta_i} + \alpha_i^{\beta_i \gamma} \tau_i^{\beta_i}}{(\tau_i^{\beta_i} + \alpha_i^{\beta_i})^{\gamma + 1}} \right)^2 \\
	&= \left( -\beta_i \boldsymbol{x}_i, \log \frac{\tau_i}{\alpha_i} \beta_i \boldsymbol{x}_i \right)^2 K_i \frac{\tau_i^{\beta_i}}{\tau_i^{\beta_i} + \alpha_i^{\beta_i}} \left(1 - \frac{\tau_i^{\beta_i}}{\tau_i^{\beta_i} + \alpha_i^{\beta_i}} \right) \left( \frac{\tau_i^{\beta_i \gamma} \alpha_i^{\beta_i} + \alpha_i^{\beta_i \gamma} \tau_i^{\beta_i}}{(\tau_i^{\beta_i} + \alpha_i^{\beta_i})^{\gamma + 1}} \right)^2 \\
	&=  \begin{pmatrix}
		\beta_i^2 \boldsymbol{x_i} \boldsymbol{x_i}^T & -\log \frac{\tau_i}{\alpha_i} \beta_i^2 \boldsymbol{x_i} \boldsymbol{x_i}^T\\
		-\log \frac{\tau_i}{\alpha_i} \beta_i^2 \boldsymbol{x_i} \boldsymbol{x_i}^T & \left( \log \frac{\tau_i}{\alpha_i} \beta_i \right)^2  \boldsymbol{x_i} \boldsymbol{x_i}^T
	\end{pmatrix} K_i \left( \frac{\tau_i^{\beta_i}}{\tau_i^{\beta_i} + \alpha_i^{\beta_i}} \right)^3 \left(\frac{\alpha_i^{\beta_i}}{\tau_i^{\beta_i} + \alpha_i^{\beta_i}} \right)^3 \\
	&\quad \left( \frac{\tau_i^{\beta_i (\gamma-1)} \alpha_i^{\beta_i} + \alpha_i^{\beta_i (\gamma-1)} \tau_i^{\beta_i}}{(\tau_i^{\beta_i} + \alpha_i^{\beta_i})^{\gamma - 1}} \right)^2 \\
	&=  \sum_{i=1}^I K_i \left( F_{\alpha_i, \beta_i}(\tau_i, \boldsymbol{x_i})^{\gamma-1} + R_{\alpha_i, \beta_i}(\tau_i, \boldsymbol{x_i})^{\gamma-1} \right)^2 F_{\alpha_i, \beta_i}(\tau_i, \boldsymbol{x_i})^3 \cdot R_{\alpha_i, \beta_i}(\tau_i, \boldsymbol{x_i})^3  \cdot \boldsymbol{M}_i \\
	&= K \cdot \boldsymbol{K}_\gamma(\boldsymbol{\theta}).
\end{align*}
Therefore, we have that
\[
\sqrt{K} \boldsymbol{U}_{\gamma}(\tilde{\boldsymbol{\theta}}_{\gamma}) \underset{K \rightarrow \infty}{\overset{ \mathcal{L}}{\to}} N\biggr(\boldsymbol{0}_{2(J+1)},  \boldsymbol{K}_{\gamma}(\boldsymbol{\theta}_0) \biggr).
\]
From this result, it follows that
\[
\sqrt{K}  \boldsymbol{Q}_{\gamma}(\boldsymbol{\theta}_0)^T \boldsymbol{U}_{\gamma}(\tilde{\boldsymbol{\theta}}_{\gamma})  \underset{K \rightarrow \infty}{\overset{ \mathcal{L}}{\to}} N\biggr(\boldsymbol{0}_{2(J+1)}, \boldsymbol{Q}_{\gamma}(\boldsymbol{\theta}_0)^T \boldsymbol{K}_{\gamma}(\boldsymbol{\theta}_0) \boldsymbol{Q}_{\gamma}(\boldsymbol{\theta}_0)\biggr).
\]
Since 
\(
\boldsymbol{Q}_{\gamma}(\boldsymbol{\theta}_0)^T \boldsymbol{K}_{\gamma}(\boldsymbol{\theta}_0) \boldsymbol{Q}_{\gamma}(\boldsymbol{\theta}_0) \biggr(\boldsymbol{Q}_{\gamma}(\boldsymbol{\theta}_0)^T \boldsymbol{K}_{\gamma}(\boldsymbol{\theta}_0) \boldsymbol{Q}_{\gamma}(\boldsymbol{\theta}_0)\biggr)^{-1} = \boldsymbol{I},
\)
we can conclude that
\[
K \boldsymbol{U}_{\gamma}(\tilde{\boldsymbol{\theta}}_{\gamma})^T \boldsymbol{Q}_{\gamma}(\boldsymbol{\theta}_0) \biggr(\boldsymbol{Q}_{\gamma}(\boldsymbol{\theta}_0)^T \boldsymbol{K}_{\gamma}(\boldsymbol{\theta}_0) \boldsymbol{Q}_{\gamma}(\boldsymbol{\theta}_0)\biggr)^{-1} \boldsymbol{Q}_{\gamma}(\boldsymbol{\theta}_0)^T \boldsymbol{U}_{\gamma}(\tilde{\boldsymbol{\theta}}_{\gamma}) \underset{K \rightarrow \infty}{\overset{ \mathcal{L}}{\to}} \chi^2_r .
\]

\end{document}